\documentclass{custom}

\newcommand*\txte{\text{e}}
\newcommand*\proj[1]{\tfrac{#1#1^*}{#1^*#1}}
\newcommand*\projM[1]{\tfrac{#1#1^*M}{#1^*M#1}}

\title{Generalized Davidson and multidirectional-type methods for the
generalized singular value decomposition\thanksNwo}
\author{Ian N. Zwaan\thanks{%
  Department of Mathematics and Computer Science, TU Eindhoven, PO Box
  513, NL-5600 MB Eindhoven, The Netherlands,
  \{\texttt{ianzwaan.com}, \texttt{www.win.tue.nl/\allowbreak
  {\textasciitilde}hochsten}\}.}\, and Michiel E. Hochstenbach\sameThanks}
\date{May 12, 2017}

\begin{document}

\maketitle


\begin{abstract}
  We propose new iterative methods for computing nontrivial extremal
  generalized singular values and vectors. The first method is a
  generalized Davidson-type algorithm and the second method employs a
  multidirectional subspace expansion technique.  Essential to the
  latter is a fast truncation step designed to remove a low quality
  search direction and to ensure moderate growth of the search space.
  Both methods rely on thick restarts and may be combined with two
  different deflation approaches.  We argue that the methods have
  monotonic and (asymptotic) linear convergence, derive and discuss
  locally optimal expansion vectors, and explain why the fast truncation
  step ideally removes search directions orthogonal to the desired
  generalized singular vector.  Furthermore, we identify the relation
  between our generalized Davidson-type algorithm and the
  Jacobi--Davidson algorithm for the generalized singular value
  decomposition.  Finally, we generalize several known convergence
  results for the Hermitian eigenvalue problem to the Hermitian positive
  definite generalized eigenvalue problem.  Numerical experiments
  indicate that both methods are competitive.
\end{abstract}





\section{Introduction}
\label{tcgsvd:sec:int}

The generalized singular value decomposition (GSVD) \cite{PS81gsvd} is a
generalization of the standard singular value decomposition (SVD), and
is used in, for example, linear discriminant analysis
\autocite{HJP03ldagsvd}, the method of particular solutions
\autocite{TM08mps}, general form Tikhonov regularization
\cite[Sec.~5.1]{PCH98RANK}, and more \cite{ZB92gsvd}.  Computing the
full GSVD with direct methods can be prohibitively time-consuming for
large problem sizes; however, for many applications is suffices to
compute only a few of the largest or smallest generalized singular
values and vectors. As a result, iterative methods may become attractive
when the matrices involved are large and sparse.

An early iterative approach based on a modified Lanczos method was
introduced by \textcite{HZ96gsvd}, and later a variation by
\textcite{KHE07gsvdtik}.  Both methods are inner-outer methods that
require the solution to a least squares problem in each iteration, which
may be computationally expensive. An approach that naturally allows for
inexact solutions is the Jacobi--Davidson-type method (JDGSVD)
introduced in \autocite{MH09jdgsvd}; however, this is still an
inner-outer method.  Alternatives to the previously mentioned methods
include iterative methods designed for (symmetric positive definite)
generalized eigenvalue problems, in particular generalized Davidson
\autocite{ET00MM,ET00Kny} and LOBPCG \autocite{AK01lobpcg}. These
methods compute only the right generalized singular vectors and require
additional steps to determine the left generalized singular vectors.
More importantly, applying these methods involves squaring potentially
ill-conditioned matrices.

In this paper we discuss two new and competitive iterative methods for
the computation of extremal generalized singular values and
corresponding generalized singular vectors. The first can be seen as a
generalized Davidson-type algorithm for the GSVD, while the second
method builds upon the first, but uses multidirectional subspace
expansion alongside a fast subspace truncation. The multidirectional
subspace expansion is intended to produce improved search directions,
whereas the subspace truncation is designed to remove low-quality search
directions that are ideally orthogonal to the desired generalized
singular vector.  Both methods can be used to compute either the
smallest or the largest generalized singular values of a matrix pair, or
to approximate the truncated GSVD (TGSVD).  A crucial part of both
methods is a thick restart that allows for the removal of unwanted
elements.

The remainder of this text is organized as follows.  We
derive a generalized Davidson-type algorithm for the GSVD in the next
section, and prove multiple related theoretical properties.
We subsequently discuss a $B^*B$-orthonormal version of the algorithm
and its connection to JDGSVD in Section~\ref{tcgsvd:sec:bborth}.
In Section~\ref{tcgsvd:sec:mdgsvd}, we examine locally optimal search
directions and argue for a multidirectional subspace expansion followed
by a fast subspace truncation; then we present our second algorithm. In
Section~\ref{tcgsvd:sec:tgsvd}, we explore the deflation of generalized
singular values and generalized singular vectors. We generalize several
known error bounds for the Hermitian eigenvalue problem to results for
the generalized singular value decomposition in
Section~\ref{tcgsvd:sec:err}. Finally, we consider numerical examples
and experiments in Section~\ref{tcgsvd:sec:num}, and end with
conclusions in Section~\ref{tcgsvd:sec:con}.



\section{Generalized Davidson for the GSVD}
\label{tcgsvd:sec:gdgsvd}

Triangular and diagonal are two closely related forms of the GSVD. The
triangular form is practical for the derivation and implementation of
our methods, while the diagonal form is particularly relevant for the
analysis. We adopt the definitions from \textcite{ZB92gsvd}, but with a
slightly more compact presentation.  Let $A$ be an $m\times n$ matrix,
$B$ a $p\times n$ matrix, and assume for the sake of simplicity that
$\Null(A) \cap \Null(B) = \{\vec 0\}$; then $\rank([A^T\, B^T]^T) = n$
and there exist unitary matrices $U$, $V$, $W$, an $m\times n$ matrix
$\Sigma_A$, a $p\times n$ matrix $\Sigma_B$, and a nonsingular
upper-triangular $n\times n$ matrix $R$ such that
\begin{equation}\label{tcgsvd:eq:gsvdTriu}
  AW = U \Sigma_A R
  \quad\text{and}\quad
  BW = V \Sigma_B R.
\end{equation}
The matrices $\Sigma_A$ and $\Sigma_B$ satisfy
\begin{equation*}
  \Sigma_A^T \Sigma_A = \diag(c_1^2, \dots, c_n^2),
  \qquad
  \Sigma_B^T \Sigma_B = \diag(s_1^2, \dots, s_n^2),
  \qquad \Sigma_A^T\Sigma_A + \Sigma_B^T\Sigma_B = I,
\end{equation*}
and can be partitioned as 
\begin{equation*}
  \bordermatrix[{[]}]{%
          & l & (n-p)_+ & (n-m)_+ \cr
  l       & D_A & 0 & 0 \cr
  (n-p)_+ & 0 & I & 0 \cr
  (m-n)_+ & 0 & 0 & 0
  },
  \qquad
  \bordermatrix[{[]}]{%
          & l & (n-p)_+ & (n-m)_+ \cr
  l       & D_B & 0 & 0 \cr
  (n-m)_+ & 0 & 0 & I \cr
  (p-n)_+ & 0 & 0 & 0
  },
\end{equation*}
where $l = \min\{m,p,n,m+p-n\}$, $(\cdot)_+ = \max\{\cdot,0\}$, and
$D_A$ and $D_B$ are diagonal matrices with nonnegative entries. The
generalized singular pairs $(c_j,s_j)$ are nonnegative and define the
regular generalized singular values $\sigma_j = \infty$ if $s_j=0$ and
$\sigma_j = c_j / s_j$ otherwise. Hence, we call a generalized singular
pair $(c_j,s_j)$ large if $\sigma_j$ is large and small if $\sigma_j$ is
small, and additionally refer to the largest and smallest $\sigma_j$ as
$\sigma_{\max}$ and $\sigma_{\min}$, respectively.  The diagonal
counterpart of \eqref{tcgsvd:eq:gsvdTriu} is
\begin{equation}\label{tcgsvd:eq:gsvdDiag}
  A X = U \Sigma_A
  \quad\text{and}\quad
  B X = V \Sigma_B
  \quad\text{with}\quad
  X = W R^{-1},
\end{equation}
and is useful because the columns of $X$ are the (right) singular
vectors $\vec x_j$ and satisfy, for instance,
\begin{equation}\label{tcgsvd:eq:gsvdgep}
  s_j^2 A^*\!A \vec x_j = c_j^2 B^*B \vec x_j.
\end{equation}
The assumption $\Null(A) \cap \Null(B) = \{\vec 0\}$ is not necessary
for the implementation of our algorithm; nevertheless, we will make this
assumption for the remainder of the text to simplify our
discussion and analysis.  We may also assume without loss of generality
that the desired generalized singular values are contained in the
leading principal submatrices of the factors.  Consequently, if $k < l$
and $C_k$, $S_k$, and $R_k$ denote the leading $k\times k$ principal
submatrices of $\Sigma_A$, $\Sigma_B$, and $R$; and $U_k$, $V_k$, $W_k$,
and $X_k$ denote the first $k$ columns of $U$, $V$, $W$, and $X$;
then $X_k = W_k R_k^{-1}$ and we can define the partial (or truncated)
GSVD of $(A,B)$ as
\begin{equation*}
    AW_k = U_k C_k R_k
    \quad\text{and}\quad
    BW_k = V_k S_k R_k.
\end{equation*}
We aim is to approximate this partial GSVD for a $k \ll n$.

Since \eqref{tcgsvd:eq:gsvdgep} can be interpreted as a generalized
eigenvalue problem, it appears reasonable to consider the search space
\begin{multline*}
  \subspace W_k
  = \linspan \{ \vec{\widetilde x}_{(0)},
  (\widetilde s_{(0)}^2 A^*\!A - \widetilde c_{(0)}^2 B^*B) \vec{\widetilde x}_{(0)}, \\
  (\widetilde s_{(1)}^2 A^*\!A - \widetilde c_{(1)}^2 B^*B) \vec{\widetilde x}_{(1)}, \dots,
  (\widetilde s_{(k-1)}^2 A^*\!A - \widetilde c_{(k-1)}^2 B^*B)
  \vec{\widetilde x}_{(k-1)} \},
\end{multline*}
consisting of homogeneous residuals generated by the generalized
Davidson method (c.f., e.g., \cite[Sec.~11.2.4]{ET00MM} and
\cite[Sec.~11.3.6]{ET00Kny}) applied to the matrix pencil
$(A^*\!A,B^*B)$.  The quantities $\vec{\widetilde x}_{(j)}$, $\widetilde
c_{(j)}$, and $\widetilde s_{(j)}$ are approximations to $\vec x_1$,
$c_1$, and $s_1$ with respect to the search space $\subspace W_j$.  The
challenge is to compute a basis $W_k$ with orthonormal columns for
$\subspace W_k$ without using the products $A^*\!A$ and $B^*B$; however,
let us focus on the extraction phase first. We will later see that a
natural subspace expansion follows as a consequence.

Given $W_k$, we can compute the reduced QR decompositions
\begin{equation}\label{tcgsvd:eq:AWBW}
  AW_k = U_k H_k,
  \qquad
  BW_k = V_k K_k,
\end{equation}
where $U_k$ and $V_k$ have $k$ orthonormal columns and $H_k$
and $K_k$ are $k\times k$ and upper-triangular.
To compute the approximate generalized singular values,
let the triangular form GSVD of $(H_k,K_k)$ be given by
\begin{equation*}
  H_k \widetilde W = \widetilde U \widetilde C \widetilde R,
  \qquad
  K_k \widetilde W = \widetilde V \widetilde S \widetilde R,
\end{equation*}
where $\widetilde U$, $\widetilde V$, and $\widetilde W$ are
orthonormal, $C$ and $S$ are diagonal, and $R$ is upper triangular.  At
this point, we can readily form the approximate partial GSVD
\begin{equation}\label{tcgsvd:eq:pgsvd}
  A(W_k \widetilde W) = (U_k \widetilde U) \widetilde C \widetilde R,
  \qquad
  B(W_k \widetilde W) = (V_k \widetilde V) \widetilde S \widetilde R,
\end{equation}
and determine the leading approximate generalized singular values and
vectors. When the dimension of the search space $\subspace W_k$
grows large, a thick restart can be performed by partitioning the
decompositions in \eqref{tcgsvd:eq:pgsvd} as
\begin{equation}\label{tcgsvd:eq:part}
  \begin{split}
    A \begin{bmatrix} W_k \widetilde W_1 & W_k \widetilde W_2 \end{bmatrix}
    &= \begin{bmatrix} U_k \widetilde U_1 & U_k \widetilde U_2 \end{bmatrix}
      \begin{bmatrix} \widetilde C_{1} \\ & \widetilde C_{2} \end{bmatrix}
      \begin{bmatrix} \widetilde R_{11} & \widetilde R_{12} \\ &
      \widetilde R_{22} \end{bmatrix}, \\
    B \begin{bmatrix} W_k \widetilde W_1 & W_k \widetilde W_2 \end{bmatrix}
    &= \begin{bmatrix} V_k \widetilde V_1 & V_k \widetilde V_2 \end{bmatrix}
      \begin{bmatrix} \widetilde S_{1} \\ & \widetilde S_{2} \end{bmatrix}
      \begin{bmatrix} \widetilde R_{11} & \widetilde R_{12} \\ &
      \widetilde R_{22} \end{bmatrix},
  \end{split}
\end{equation}
and truncating to
\begin{equation*}
  A(W_k \widetilde W_1)
  = (U_k \widetilde U_1) \widetilde C_{1} \widetilde R_{11},
  \qquad
  B(W_k \widetilde W_1)
  = (V_k \widetilde V_1) \widetilde S_{1} \widetilde R_{11}.
\end{equation*}
If there is need to reorder the $c_j$ and $s_j$, then we can simply use
the appropriate permutation matrix $P$ and compute
\begin{equation*}
  \begin{split}
    A (W_k \widetilde W Q) &= (U_k \widetilde U P) (P^* C P) (P^* R Q), \\
    B (W_k \widetilde W Q) &= (V_k \widetilde V P) (P^* S P) (P^* R Q),
  \end{split}
\end{equation*}
where $Q$ is unitary and such that $P^*RQ$ is upper triangular.

For a subsequent generalized Davidson-type expansion of the search
space, let
\begin{equation*}
  \vec{\widetilde u}_1 = U_k \widetilde U_1 \vec e_1,
  \qquad
  \vec{\widetilde v}_1 = V_k \widetilde V_1 \vec e_1,
  \qquad
  \vec{\widetilde w}_1 = W_k \widetilde W_1 \vec e_1,
  \quad\text{and}\quad
  \vec{\widetilde x}_1 = \vec{\widetilde w}_1 / \widetilde r_{11}
\end{equation*}
be the approximate generalized singular vectors satisfying
\begin{equation*}
  A\vec{\widetilde x}_1 = \widetilde c_1 \vec{\widetilde u}_1
  \quad\text{and}\quad
  B\vec{\widetilde x}_1 = \widetilde s_1 \vec{\widetilde v}_1.
\end{equation*}
Then the homogeneous residual given by
\begin{equation}\label{tcgsvd:eq:gdres}
  \vec r
  = (\widetilde s_1^2 A^*\!A - \widetilde c_1^2 B^*B) \vec{\widetilde x}_1
  = \widetilde c_1 \widetilde s_1 (\widetilde s_1 A^*
    \vec{\widetilde u}_1 - \widetilde c_1 B^*\vec{\widetilde v}_1)
\end{equation}
suggests the expansion vector $\vec{\widetilde r} = \widetilde s_1
A^*\vec{\widetilde u}_1 - \widetilde c_1 B^*\vec{\widetilde v}_1$,
which is orthogonal to $W_k$.
The residual norm $\|\vec r\|$ goes to zero as the generalized singular
value and vector approximations converge, and we recommend terminating
the iterations when the right-hand side of
\begin{equation}\label{tcgsvd:eq:cnvtol}
  \frac{\|\vec r\|}{(\widetilde s_1^2 \|A^*\!A\| + \widetilde c_1^2
  \|B^*B\|)\|\vec{\widetilde x}_1\|}
  \le \frac{\sqrt{n}\,|\widetilde r_{11}|\, \|\vec r\|}{\widetilde s_1^2 \|A^*\!A\|_1 +
  \widetilde c_1^2 \|B^*B\|_1}
\end{equation}
is sufficiently small.  The left-hand side is  the normwise backward
error by \textcite{FT00bwe}, and the right-hand side is an alternative
that can be approximated efficiently; for example, using the
\texttt{normest1} function in MATLAB, which  does not require computing
the matrix products $A^*\!A$ and $B^*B$ explicitly.  The GDGSVD
algorithm is summarized in Algorithm~\ref{tcgsvd:alg:gdgsvd}.


\begin{algorithm}[Generalized Davidson for the GSVD
  (GDGSVD)]\label{tcgsvd:alg:gdgsvd}
\strut\\*
  \textbf{Input:} Matrix pair $(A,B)$, starting vector $\vec w_0$,
    minimum and maximum dimensions $j < \ell$. \\
  \textbf{Output:} $AW_j = U_j C_j R_j$ and $BW_j = V_j S_j R_j$
    approximating a partial GSVD. \\
  \tab[\phantom01.] Let $\vec{\widetilde r} = \vec w_0$. \\
  \tab[\phantom02.] \textbf{for} \textit{number of restarts}
    \textbf{and} \textit{not converged (cf., e.g.,~\eqref{tcgsvd:eq:cnvtol})} \textbf{do} \\
  \tab[\phantom03.]\tab  \textbf{for} $k = 1$, 2, \dots, $\ell$
    \textbf{do} \\
  \tab[\phantom04.]\tab\tab  $\vec w_k = \vec{\widetilde r} / \|\vec{\widetilde r}\|$. \\
  \tab[\phantom05.]\tab\tab  Update $AW_k = U_k H_k$ and $BW_k = V_k K_k$.\\
  \tab[\phantom06.]\tab\tab  Compute $H_k = \widetilde U \widetilde C
    \widetilde R \widetilde W^*$ and $K_k = \widetilde V \widetilde S
    \widetilde R \widetilde W^*$. \\
  \tab[\phantom07.]\tab\tab  Let $\vec{\widetilde r} = s_{1}
    A^*\vec{\widetilde u}_1 - c_{1} B^*\vec{\widetilde v}_1$.\\
  \tab[\phantom08.]\tab\tab  \textbf{if} $j \le k$ \textbf{and}
    \textit{converged (cf., e.g.,~\eqref{tcgsvd:eq:cnvtol})} 
    \textbf{then} \textbf{break} \\
  \tab[\phantom09.]\tab  \textbf{end} \\
  \tab[10.]\tab  Partition $\widetilde U$, $\widetilde V$, $\widetilde
  W$, $\widetilde C$, $\widetilde S$, and $\widetilde R$ according to
  \eqref{tcgsvd:eq:part}. \\
  \tab[11.]\tab  Let $U_j = U_k \widetilde U_1$, $V_j = V_k \widetilde
    V_1$, and $W_j = W_k \widetilde W_1$. \\
  \tab[12.]\tab Let $H_j = \widetilde C_{1} \widetilde R_{11}$ and $K_j
    = \widetilde S_{1} \widetilde R_{11}$. \\
  \tab[13.]  \textbf{end}
\end{algorithm}


By design, the largest (or smallest) Ritz values are preserved after
the restart; moreover, the generalized singular values increase (or
decrease) monotonically per iteration as indicated by the proposition
below.  We wish to emphasize that the proof of the proposition does not
require $B^*B$ to be nonsingular, as opposed to the Courant--Fischer
minimax principles for the generalized eigenvalue problem.



\begin{proposition}\label{tcgsvd:thm:mono}
  Let $\subspace W_k$ and $\subspace W_{k+1}$ be subspaces of
  dimensions $k$ and $k+1$, respectively, and such that $\subspace
  W_k \subset \subspace W_{k+1}$.  If $\sigma_{\max}(\subspace W)$
  and $\sigma_{\min}(\subspace W)$ denote the maximum and minimum
  generalized singular values of $A$ and $B$ with respect to the
  subspace $\subspace W$, then
  \begin{equation*}
    \sigma_{\max}
    \ge \sigma_{\max}(\subspace W_{k+1})
    \ge \sigma_{\max}(\subspace W_k)
    \ge \sigma_{\min}(\subspace W_k)
    \ge \sigma_{\min}(\subspace W_{k+1})
    \ge \sigma_{\min}.
  \end{equation*}
\end{proposition}
\begin{proof}
  Both $A^*\!A$ and $B^*B$ may be singular; therefore, we consider the
  pencil
  \begin{equation*}
    (A^*\!A, A^*\!A + B^*B) = (A^*\!A, X^{-*}X^{-1})
  \end{equation*}
  with generalized eigenvalues $c_i^2$ and note that $\sigma_i^2 = c_i^2
  / (1 - c_i^2)$ with the convention that $1/0 = \infty$. Applying the
  Courant--Fischer minimax principles yields
  \begin{equation*}
    c_1
    \ge \max_{\vec 0 \neq \vec w \in \subspace W_{k+1}}
      \frac{\|A\vec w\|}{\|X^{-1}\vec w\|}
    \ge \max_{\vec 0 \neq \vec w \in \subspace W_{k}}
      \frac{\|A\vec w\|}{\|X^{-1}\vec w\|}
    \ge \min_{\vec 0 \neq \vec w \in \subspace W_{k}}
      \frac{\|A\vec w\|}{\|X^{-1}\vec w\|}
    \ge \min_{\vec 0 \neq \vec w \in \subspace W_{k+1}}
      \frac{\|A\vec w\|}{\|X^{-1}\vec w\|}
    \ge c_n.
  \end{equation*}
\end{proof}


Proposition~\ref{tcgsvd:thm:mono} implies that if a basis $W_k$ for a
subspace $\subspace W_k$ is computed by
Algorithm~\ref{tcgsvd:alg:gdgsvd}, then
\begin{equation*}
  \sigma_{\max}(\subspace W_k)
  = \max_{\vec 0 \neq \vec w \in \subspace W_{k}}
    \frac{\|A\vec w\|}{\|X^{-1}\vec w\|}
  = \max_{\vec c \neq \vec 0}
    \frac{\|AW_k \vec c\|}{\|[A^T\;B^T]^T W_k \vec c\|}
  = \max_{\vec c \neq \vec 0}
    \frac{\|H_k \vec c\|}{\|[H_k^T\;K_k^T]^T \vec c\|};
\end{equation*}
that is, the largest generalized singular value of the matrix pair
$(A,B)$ with respect to the subspace $\subspace W_k$ is the largest
generalized singular value of $(H_k,K_k)$. A similar statement
holds for the smallest generalized singular value.  Furthermore, the
matrix pair $(H_k,K_k)$ is optimal in the sense of the following
proposition.


\begin{proposition}\label{tcgsvd:thm:optHK}
  Let the $M$-Frobenius norm for a Hermitian positive definite matrix
  $M$ be defined as $\|Y\|_{F,M}^2 = \operatorname{trace}(Y^*MY)$. Now
  consider the decompositions from \eqref{tcgsvd:eq:AWBW} and define the
  residuals
  \begin{equation*}\label{tcgsvd:eq:4res}
    \begin{split}
      R_1(G) &= AW_k - U_k G, \\
      R_2(G) &= BW_k - V_k G, \\
      R_3(G) &= A^*U_k - B^*V_k G^*, \\
      R_4(G) &= B^*V_k - A^*U_k G^*;
    \end{split}
  \end{equation*}
  then the following results hold.
  \begin{enumerate}
    \item $G = H_k = U_k^* A W_k$ minimizes $\|R_1(G)\|_2$ and
      is the unique minimizer of $\|R_1(H_k)\|_F$.
    \item $G = K_k = V_k^* B W_k$ minimizes $\|R_2(G)\|_2$ and
      is the unique minimizer of $\|R_2(K_k)\|_F$.
    \item If $B^*B$ is nonsingular, then $G = H_k K_k^{-1}$
      minimizes $\|R_3(G)\|_{(B^*B)^{-1}}$ and is the unique minimizer
      of $R_3$ with respect to the $(B^*B)^{-1}$-Frobenius norm.
    \item If $A^*\!A$ is nonsingular, then $G = K_k H_k^{-1}$
      minimizes $\|R_4(G)\|_{(A^*\!A)^{-1}}$ and is the unique minimizer
      of $R_4$ with respect to the $(A^*\!A)^{-1}$-Frobenius norm.
  \end{enumerate}
\end{proposition}
\begin{proof}
  With the observation that $A^*U_k = A^*\!AW_k H_k^{-1}$ and
  $B^*V_k = B^*BW_k K_k^{-1}$, the proof becomes a
  straightforward adaptation of \autocite[Thm~2.1]{MH09jdgsvd}.
\end{proof}


Propositions~\ref{tcgsvd:thm:mono} and \ref{tcgsvd:thm:optHK}
demonstrate that the convergence behavior of
Algorithm~\ref{tcgsvd:alg:gdgsvd} is monotonic, and that the computed
$H_k$ and $K_k$ are in some sense optimal for the search space
$\subspace W_k = \linspan(W_k)$; however, the propositions make no
statement regarding the quality of the subspace expansion. A locally
optimal residual-type subspace expansion can be derived with inspiration
from \textcite{QY08opex}.


\begin{proposition}\label{tcgsvd:thm:optRx}
  Define
  \begin{equation*}
    R_k = A^*\!AW_k (H_k^*H_k + K_k^*K_k)^{-1} K_k^*K_k
      {} - B^*BW_k (H_k^*H_k + K_k^*K_k)^{-1}H_k^*H_k
  \end{equation*}
  and let $\vec r = R_k\vec c$; then
  \begin{equation*}
    \cos^2 (\vec x_1, [W_k\; \vec r])
    = \cos^2 (\vec x_1, W_k) + \cos^2 (\vec x_1, \vec r)
  \end{equation*}
  is maximized for $\vec{c} = R_k^+ \vec x_1$.
\end{proposition}
\begin{proof}
  Since $\Null(A) \cap \Null(B) = \{ \vec 0 \}$ we also have 
  $\Null(H_k) \cap \Null(K_k) = \{ \vec 0 \}$, which implies that
  $H_k^*H_k + K_k^*K_k$ is invertible and $R_k$ is
  well-defined. Furthermore, it is now straightforward to verify that
  \begin{equation*}
    W_k^*R_k
    = H_k^*H_k (H_k^*H_k + K_k^*K_k)^{-1} K_k^*K_k
    {} - K_k^*K_k (H_k^*H_k + K_k^*K_k)^{-1} H_k^*H_k = 0
  \end{equation*}
  using the GSVD of $H_k$ and $K_k$. It follows that
  \begin{equation*}
    \|[W_k\; \vec r]^*\vec x_1\|^2
    = \|W_k^*\vec x_1\|^2 + |\vec r^*\vec x_1|^2,
  \end{equation*}
  which realizes its maximum for $\vec c = R_k^+ \vec x_1$.
\end{proof}



Different choices for $R_k$ in Proposition~\ref{tcgsvd:thm:optRx} are
possible; however, the current choice does not require additional
assumptions on, for instance, $H_k$ and $K_k$. Regardless of the
choice of $R_k$, computing the optimal expansion vector is generally
impossible without a priori knowledge of the desired generalized
singular vector $\vec x_1$.  Therefore, we expand the search space with
a residual-type vector similar to generalized Davidson. The convergence
of generalized Davidson is closely connected to steepest descent and has
been studied extensively; see, for example, \textcite{EO03gd1,EO03gd2}
and references therein. For completeness, we add the following
asymptotic bound for the GSVD.



\begin{proposition}\label{tcgsvd:thm:asymconv}
  Let $(c_1,s_1)$ be the smallest generalized singular pair of $(A,B)$
  with corresponding generalized singular vector $\vec x_1$, and assume
  the pair is simple.  Define the Hermitian positive definite operator
  $M = s_1^2 A^*\!A - c_1^2 B^*B$ restricted to the domain perpendicular
  to $(A^*\!A + B^*B)\vec x_1 = X^{-*} \vec e_1$, and let the
  eigenvalues of $M$ be given by
  \begin{equation*}
    \lambda_1 \ge \lambda_2 \ge \dots \ge \lambda_{n-1} > 0.
  \end{equation*}
  Furthermore, let $\vec{\widetilde x}_1$, $\widetilde c_1 =
  \|A\vec{\widetilde x}_1\|$, and $\widetilde s_1 = \|B\vec{\widetilde
  x}_1\|$ approximate $\vec x_1$, $c_1$, and $s_1$, respectively, and be
  such that $\widetilde c_1^2 + \widetilde s_1^2 = 1$.  If
  $\vec{\widetilde x}_1 = \xi \vec x_1 + \vec f$ for some scalar $\xi$
  and vector $\vec f \perp X^{-*}\vec e_1$; then
  \begin{equation*}
    \sin^2([\vec{\widetilde x}_1\; \vec r], \vec x_1)
    \le \left(\frac{\kappa-1}{\kappa+1}\right)^2
    \sin^2(\vec{\widetilde x}_1, \vec x_1)
    {} + \Landau O(\|\vec f\|^3),
  \end{equation*}
  where $\kappa = \lambda_{1} / \lambda_{n-1}$ is the condition number
  of $M$, and $\vec r = (\widetilde s_1^2 A^*\!A - \widetilde c_1^2
  B^*B) \vec{\widetilde x}_1$
  is the homogeneous residual.
\end{proposition}
\begin{proof}
  We have
  \begin{equation*}
    \widetilde c_1^2
    = \vec{\widetilde x}_1^*A^*\!A\vec{\widetilde x}_1
    = \xi^2 c_1^2 + \|A\vec f\|^2
    \quad\text{and}\quad
    \widetilde s_1^2
    = \vec{\widetilde x}_1^*B^*B\vec{\widetilde x}_1
    = \xi^2 s_1^2 + \|B\vec f\|^2,
  \end{equation*}
  and it follows that
  \begin{equation*}
      \vec r
      = \xi^2 (s_1^2 A^*\!A - c_1^2 B^*B) \vec f
      {} + (\|B\vec f\|^2 A^*\!A - \|A\vec f\|^2 B^*B)
      (\xi \vec x_1 + \vec f)
      = \xi^2 M \vec f + \Landau O(\|\vec f\|^2)
  \end{equation*}
  and
  \begin{equation*}
    \vec r^*\vec x_1
    = \xi \vec f^* (s_1^2 A^*\!A - c_1^2 B^*B) \vec f
    = \xi \vec f^* M \vec f.
  \end{equation*}
  Hence, $\vec{\widetilde x}_1 = \vec x_1$ if $\|\vec r\| = 0$ for
  $\vec{\widetilde x}_1$ sufficiently close to $\vec x_1$ and we are
  done. Otherwise, $\vec r$ is nonzero and perpendicular to
  $\vec{\widetilde x}_1$, so that
  \begin{equation*}
    \sin^2([\vec{\widetilde x}_1\; \vec r],\vec x_1)
    = 1 - \cos^2(\vec{\widetilde x}_1,\vec x_1) - \cos^2(\vec r,\vec x_1)
    = \left(
        1 - \frac{\cos^2(\vec r,\vec x_1)}{\sin^2(\vec{\widetilde x}_1,\vec x_1)}
      \right) \sin^2(\vec{\widetilde x}_1, \vec x_1).
  \end{equation*}
  Combining the above expressions, and using the fact that nontrivial
  orthogonal projectors have unit norm, yields
  \begin{equation*}
    \frac{\cos^2(\vec r,\vec x_1)}{\sin^2(\vec{\widetilde x}_1,\vec x_1)}
    = \frac{|\vec r^*\vec x_1|^2}{\|\vec r\|^2
      \|(I - \vec{\widetilde x}_1 \vec{\widetilde x}_1^*)\vec x_1\|^2}
    \ge \frac{|\vec f^* M \vec f|^2}{%
      \|M \vec f\|^2\|\vec f\|^2} + \Landau
      O(\|\vec f\|).
  \end{equation*}
  Using the Kantorovich inequality (cf., e.g.,~\autocite[p.~68]{GR97nr})
  we obtain
  \begin{equation*}
    \begin{split}
      \sin^2([\vec{\widetilde x}_1\; \vec r],\vec x_1)
      &\le \left(
        1 - \frac{4\lambda_1\lambda_{n-1}}{(\lambda_1+\lambda_{n-1})^2}
      \right) \sin^2(\vec{\widetilde x}_1, \vec x_1)
      {} + \Landau O(\|\vec f\|\sin^2(\vec{\widetilde x}_1, \vec x_1)) \\
      &= \left(\frac{\kappa - 1}{\kappa + 1} \right)^2
      \sin^2(\vec{\widetilde x}_1, \vec x_1)
      {} + \Landau O(\|\vec f\|\sin^2(\vec{\widetilde x}_1, \vec x_1)).
    \end{split}
  \end{equation*}
  Finally, $\widetilde c_1^2 + \widetilde s_1^2 = 1$ implies
  $\|\vec{\widetilde x}_1\| \ge \sigma_{\min}(X)$, so that
  $\sin(\vec{\widetilde x}_1, \vec x) \le \sigma_{\min}^{-1}(X)\,
  \|\vec f\| = \Landau O(\|\vec f\|)$.
\end{proof}


The condition number $\kappa$ from Proposition~\ref{tcgsvd:thm:asymconv}
may be large in practice, in which case the quantity
$(\kappa-1)/(\kappa+1)$ is close to 1. However, this upper bound may be
rather pessimistic and we will see considerably faster convergence
during the numerical tests in Section~\ref{tcgsvd:sec:num}.



\section[\texorpdfstring{$B^*B$}{B\^{}*B}-orthonormal GDGSVD]%
{$\bm{B^*B}$-orthonormal GDGSVD}
\label{tcgsvd:sec:bborth}


In the previous section we have derived the GDGSVD algorithm for an
orthonormal basis of $\subspace W_k$. An alternative is to construct a
$B^*B$-orthonormal basis of $\subspace W_k$, which allows us to use the
SVD instead of the slower GSVD for the projected problem, as well as
reduce the amount of work necessary for a restart. Another benefit is
that the $B^*B$-orthonormality reveals the connection between GDGSVD and
JDGSVD, a Jacobi--Davidson-type algorithm for the GSVD
\autocite{MH09jdgsvd}.

The derivation of $B^*B$-orthonormal GDGSVD is similar to the derivation
of Algorithm~\ref{tcgsvd:alg:gdgsvd}. Suppose that $B^*B$ is
nonsingular, let $\widehat W_k$ be a basis of $\subspace W_k$ satisfying
$\widehat W_k^* B^*B \widehat W_k = I$, and compute the QR-decomposition
\begin{equation}\label{tcgsvd:eq:BBAW}
  A\widehat W_k = \widehat U_k\widehat H_k,
\end{equation}
where $\widehat U$ has orthonormal columns and $\widehat H_k$ is
upper-triangular. Note that \eqref{tcgsvd:eq:BBAW} can be obtained from
the QR-decompositions in \eqref{tcgsvd:eq:AWBW} by setting $\widehat W_k
= W_k K_k^{-1}$, $\widehat U_k = U_k$, and $\widehat H_k = H_k
K_k^{-1}$. If $\widehat H_k = \widetilde U \Sigma \widetilde W^*$ is the
SVD of $\widehat H_k$; then
\begin{equation*}
  A(\widehat W_k \widetilde W) = (\widehat U_k \widetilde U) \Sigma,
\end{equation*}
which can be partitioned as
\begin{equation}\label{tcgsvd:eq:bbpart}
  A \begin{bmatrix} \widehat W_k \widetilde W_1 & \widehat W_k
  \widetilde W_2 \end{bmatrix}
  = \begin{bmatrix} \widehat U_k \widetilde U_1 & \widehat U_k
  \widetilde U_2 \end{bmatrix}
  \begin{bmatrix} \Sigma_{1} \\ & \Sigma_{2} \end{bmatrix}
\end{equation}
and truncated to $A\widehat W_k \widetilde W_1 = \widehat U_k
\widetilde U_1 \Sigma_{1}$.  With $\vec{\widehat u}_1 = \widehat U_k
\widetilde U \vec e_1$ and $\vec{\widehat w}_1 = \widehat W_k \widetilde
W \vec e_1$ we get the residual
\begin{equation*}
  \vec r = (A^*\!A - \sigma_1^2 B^*B)\vec{\widehat w_1}
  = \sigma_1 (A\vec{\widehat u}_1 - \sigma_1 B^*B\vec{\widehat w_1})
\end{equation*}
and the expansion vector $\vec{\widehat r} = A\vec{\widehat u}_1 -
\sigma_1 B^*B\vec{\widehat w_1}$.  The expansion vector $\vec{\widehat
r}$ is orthogonal to $\widehat W_k$ in exact arithmetic, but should in
practice still be orthogonalized with respect to $\widehat W_k$ prior to
$B^*B$-orthogonalization in order to improve numerical stability and
accuracy \autocite[Sec.~3.5]{Hig02}. Finally, in the $B^*B$-orthonormal
case the suggested stopping condition \eqref{tcgsvd:eq:cnvtol} becomes
\begin{equation}\label{tcgsvd:eq:BBcnvtol}
  \frac{\|\vec r\|}{(\|A^*\!A\| + \sigma_1^2\|B^*B\|) \|\vec{\widehat w}_1\|}
  \le \frac{\sqrt{n}\,\|\vec r\|}{(\|A^*\!A\|_1 + \sigma_1^2\|B^*B\|_1) \|\vec{\widehat w}_1\|}
  \le \tau
\end{equation}
for some tolerance $\tau$.  The algorithm is summarized below in
Algorithm~\ref{tcgsvd:alg:BBgdgsvd}, where $\widehat V_k = B \widehat
W_k$ has orthonormal columns.


\begin{algorithm}[$B^*B$-orthonormal GDGSVD]\label{tcgsvd:alg:BBgdgsvd}
\strut\\*
  \textbf{Input:} Matrix pair $(A,B)$, starting vector $\vec w_0$,
    minimum and maximum dimensions $j < \ell$. \\
  \textbf{Output:} Orthonormal $\widehat U_j$, $B^*B$-orthonormal
    $\widehat W_j$, and diagonal $\Sigma_j$ satisfying $A\widehat W_j =
    \widehat U_j \Sigma_j$. \\
  \tab[\phantom01.] Let $\widehat W_0 = \widehat V_0 = []$ and
    $\vec{\widehat r} = \vec w_0$. \\
  \tab[\phantom02.] \textbf{for} \textit{number of restarts}
    \textbf{and} \emph{not converged (cf., e.g.,~\eqref{tcgsvd:eq:BBcnvtol})}
    \textbf{do} \\
  \tab[\phantom03.] \tab \textbf{for} $k = 1$, 2, \dots, $\ell$
    \textbf{do} \\
  \tab[\phantom04.] \tab\tab $\vec{\widehat w}_k = (I - \widehat W_{k-1}
    (\widehat W_{k-1}^* \widehat W_{k-1})^{-1} \widehat W_{k-1}^*)
    \vec{\widehat r}$. \\
  \tab[\phantom05.] \tab\tab Compute $\vec{\widehat v}_k = B
    \vec{\widehat w}_k$. \\
  \tab[\phantom06.] \tab\tab $B^*B$-orthogonalize: $\vec{\widehat
    w}_k = \vec{\widehat w_k} - \widehat W_{k-1} \widehat
    V_{k-1}^* \vec{\widehat v}_k$. \\
  \tab[\phantom07.] \tab\tab $\vec{\widehat v}_k = (I - \widehat
    V_{k-1} \widehat V_{k-1}^*) \vec{\widehat v}_k$. \\
  \tab[\phantom08.] \tab\tab $\vec{\widehat w}_k = \vec{\widehat
    w}_k / \|\vec{\widehat v}_k\|$ and $\vec{\widehat v}_k =
    \vec{\widehat v}_k / \|\vec{\widehat v}_k\|$. \\
  \tab[\phantom09.] \tab\tab Update the QR-decomposition $A\widehat
    W_k = \widehat U_k \widehat H_k$. \\
  \tab[10.] \tab\tab Compute the SVD $\widehat H_k =
    \widetilde U \Sigma \widetilde W^*$. \\
  \tab[11.] \tab\tab $\vec{\widehat r} = A^* \widehat U_k
    \vec{\widetilde u}_1 - \sigma_1 B^* \widehat V_k \vec{\widetilde
    w}_1$. \\
  \tab[12.] \tab\tab \textbf{if} $j \le k$ \textbf{and}
    \textit{converged (cf., e.g.,~\eqref{tcgsvd:eq:BBcnvtol})}
    \textbf{then break} \\
  \tab[13.] \tab \textbf{end} \\
  \tab[14.] \tab Partition $\widetilde U$, $\Sigma$, and
    $\widetilde W$ according to \eqref{tcgsvd:eq:bbpart}. \\
  \tab[15.] \tab Let $\widehat U_j = \widehat U_k \widetilde U_1$,
    $\widehat V_j = \widehat V_k \widetilde W_1$, and $\widehat W_j =
    \widehat W_k \widetilde W_1$. \\
  \tab[16.] \tab Let $H_j = \Sigma_1$. \\
  \tab[17.] \textbf{end}
\end{algorithm}


The product $B^*B$ may be arbitrarily close to singularity, and a
severely ill-conditioned $B^*B$ may prove to be problematic despite the
additional orthogonalization step in
Algorithm~\ref{tcgsvd:alg:BBgdgsvd}.  Therefore, we would generally
advise against using Algorithm~\ref{tcgsvd:alg:BBgdgsvd}, and recommend
using Algorithm~\ref{tcgsvd:alg:gdgsvd} and orthonormal bases instead.
However, $B^*B$-orthonormal GDGSVD relates nicely to JDGSVD on a
theoretical level, regardless of the potential practical issues.
In JDGSVD the search spaces $\widehat U_k$ and $\widehat W_k$ are
repeatedly updated with the vectors $\vec s \perp \vec{\widehat u}_1$
and $\vec t \perp \vec{\widehat w}_1$, which are obtained by solving
correction equations. Picking the updates
\begin{equation*}
  \vec s = (I - \vec{\widehat u}_1 \vec{\widehat u}_1^*) A\vec r
  \quad\text{and}\quad
  \vec t = \vec r,
\end{equation*}
instead of solving the correction equations gives JDGSVD the same
subspace expansions as $B^*B$-orthogonal GDGSVD.  Furthermore, standard
extraction in JDGSVD is performed by computing the SVD of $\widehat
U_k^* A \widehat W_k$, which is identical to the extraction in
$B^*B$-orthonormal GDGSVD.
For harmonic Ritz extraction, JDGSVD uses the harmonic Ritz vectors
$\widehat U_k \vec c$ and $\widehat W_k \vec d$, where $\vec c$
and $\vec d$ solve
\begin{equation*}
  \widehat W_k^* A^*\!A \widehat W_k \vec d
  = \sigma^2 \widehat W_k^* B^*B \widehat W_k \vec d
  \quad\text{and}\quad
  \vec c = \sigma (\widehat W_k^* A^* \widehat U_k)^{-1}
    \widehat W_k^* B^*B \widehat W_k \vec d.
\end{equation*}
The above simplifies to
\begin{equation*}
  \widetilde W \Sigma^2 \widetilde W^* \vec d = \sigma^2 \vec d
  \quad\text{and}\quad
  \vec c = \sigma \widetilde U \Sigma^{-1} \widetilde W^* \vec d,
\end{equation*}
for $B^*B$-orthonormal GDGSVD and produces the same primitive Ritz
vectors as the standard extraction.  To summarize, JDGSVD coincides with
$B^*B$-orthonormal GDGSVD for specific expansion vectors, and there is
no difference between standard and harmonic extraction in
$B^*B$-orthonormal GDGSVD.  The difference in practice between the two
methods is primarily caused by the different expansion phases, where
GDGSVD uses residual-type vectors and JDGSVD normally solves correction
equations.  In the next section we will discuss how the subspace
expansion for GDGSVD may be further improved.



\section{Multidirectional subspace expansion}
\label{tcgsvd:sec:mdgsvd}

While the residual vector $\vec{r}$ from \eqref{tcgsvd:eq:gdres} is a
practical choice for the subspace expansion, it is not necessarily
optimal. In fact, neither is the vector given by
Proposition~\ref{tcgsvd:thm:optRx}, which is only the optimal
``residual-type'' expansion vector. In their most general form, the
desired expansion vectors are
\begin{equation}\label{tcgsvd:eq:genmd}
  \vec a - \vec b,
  \quad\text{where}\quad
  \vec a = (I - W_k W_k^*) A^*\!AW_k \vec{c_\star}
  \quad\text{and}\quad
  \vec b = (I - W_k W_k^*) B^*BW_k \vec{d_\star},
\end{equation}
for some ``optimal'' choice of $\vec{c_\star}$ and $\vec{d_\star}$. The
following proposition characterizes $\vec{c_\star}$ and $\vec{d_\star}$.


\begin{proposition}\label{tcgsvd:thm:RSex}
  Let $R_k$ and $\vec r$ be defined as in
  Proposition~\ref{tcgsvd:thm:optRx}, and assume that $R_k$ has full
  column rank. If $R_k^*A^*\!AW_k$ and $R_k^*B^*BW_k$ are nonsingular
  and if $\vec s = S_k \vec{d}$ with
  \begin{equation*}
    S_k = (A^*\!AW_k - W_k H_k^*H_k) (R_k^*A^*\!AW_k)^{-1}
    {} - (B^*BW_k - W_k K_k^*K_k) (R_k^*B^*BW_k)^{-1};
  \end{equation*}
  then
  \begin{equation*}
    \cos^2 (\vec x_1, [W_k\; \vec r\; \vec s])
    = \cos^2 (\vec x_1, W_k) + \cos^2 (\vec x_1, \vec r) + \cos^2
    (\vec x_1, \vec s)
  \end{equation*}
  is maximized for $\vec{c} = R_k^+\vec x_1$ and $\vec{d} = S_k^+ \vec
  x_1$. Moreover, for any $\vec{c}$, $\vec{d}$, and scalar $t$, the
  linear combination $R_k \vec{c} + tS_k \vec{d}$ can be written in the
  form of \eqref{tcgsvd:eq:genmd}. The mapping from $\vec{c}$ and
  $\vec{d}$ to $\vec{c_\star}$ and $\vec{d_\star}$ is one-to-one if $t
  \neq 0$.
\end{proposition}
\begin{proof}
  For the first part of the proof, use that $W_k^*R_k = W_k^*S_k =
  R_k^*S_k = 0$.  For the second part, define the shorthand $M =
  H_k^*H_k + K_k^*K_k$ and recall that
  \begin{equation*}
    R_k = A^*\!AW_k M^{-1} K_k^*K_k - B^*BW_k M^{-1} H_k^*H_k.
  \end{equation*}
  Hence, for any $\vec{c}$, $\vec{d}$ and scalar $t$ we have
  \begin{equation*}
    \begin{split}
      R_k \vec{c} + t S_k \vec{d}
      &= (I - W_k W_k^*) R_k \vec{c}
      {} + t (I - W_k W_k^*) S_k \vec{d} \\
      &= (I - W_k W_k^*) A^*\!A W_k \big(
      M^{-1} K_k^*K_k \vec{c}
      {} + t (R_k^* A^*\!A W_k)^{-1} \vec{d} \big) \\
      &{} - (I - W_k W_k^*) B^*B W_k \big(
      M^{-1} H_k^*H_k \vec{c}
      {} + t (R_k^* B^*B W_k)^{-1} \vec{d} \big) \\
      &= \vec a - \vec b,
    \end{split}
  \end{equation*}
  where $\vec a$ and $\vec b$ are defined as in \eqref{tcgsvd:eq:genmd}
  for the $\vec{c_\star}$ and $\vec{d_\star}$ satisfying
  \begin{equation*}
    \begin{bmatrix} \vec{c_\star} \\ \vec{d_\star} \end{bmatrix}
    =
    \begin{bmatrix}
      M^{-1} H_k^*H_k & t (R_k^*\!A^*\!AW_k)^{-1} \\
      M^{-1} K_k^*K_k & t (R_k^*B^*BW_k)^{-1}
    \end{bmatrix}
    \begin{bmatrix} \vec{c} \\ \vec{d} \end{bmatrix}.
  \end{equation*}
  Finally, the matrix above is invertible if
  \begin{equation*}
    t \det \begin{bmatrix}
      (R_k^*\!A^*\!AW_k)^{-1} \\ & (R_k^*B^*BW_k)^{-1}
    \end{bmatrix}
    \cdot \det \begin{bmatrix}
      R_k^*\!A^*\!AW_k M^{-1} H_k^*H_k & I \\
      R_k^*B^*BW_k M^{-1} K_k^*K_k & I
    \end{bmatrix}
    \neq 0,
  \end{equation*}
  where the first determinant is nonzero because its subblocks are
  invertible, and the second determinant equals
  \begin{equation*}
    \det(R_k^*\!A^*\!AW_k M^{-1} H_k^*H_k
    {} - R_k^*B^*BW_k M^{-1} K_k^*K_k)
    = \det(R_k^*R_k) \neq 0
  \end{equation*}
  since $R_k$ has full column rank.
\end{proof}



Let $\vec r$ and $\vec s$ be two nonzero orthogonal vectors; then the
locally optimal search direction in $\subspace S = \linspan\{\vec r,
\vec s\}$ is the projection of the desired generalized singular vector
$\vec x_1$ onto $\subspace S$, and is given by
\begin{equation}\label{tcgsvd:eq:xontors}
  \frac{\vec r^* \vec x_1}{\vec r^*\vec r} \vec r
  + \frac{\vec s^* \vec x_1}{\vec s^*\vec s} \vec s.
\end{equation}
The remaining orthogonal direction in $\subspace S$ is
\begin{equation}\label{tcgsvd:eq:xperprs}
  (\vec x_1^* \vec s) \vec r - (\vec x_1^* \vec r) \vec s,
\end{equation}
which is perpendicular to $\vec x_1$.  It is usually impossible to
compute the vectors from Proposition~\ref{tcgsvd:thm:RSex} and the
linear combination in \eqref{tcgsvd:eq:xontors} without a priori
knowledge of $\vec x_1$.  Therefore, the idea is to pick $\vec r$ and
$\vec s$ or $\vec a$ and $\vec b$ based on a different criterion, expand
the search space with both vectors, and to rely on the extraction
process to determine a good new search direction.  If successful, then
\eqref{tcgsvd:eq:xperprs} suggests that there is at least one direction
in the enlarged search space that is (nearly) perpendicular to $\vec
x_1$.  This direction may be removed to avoid excessive growth of the
search space.

For example, we could use the approximate generalized singular pair
and corresponding vectors from Section~\ref{tcgsvd:sec:gdgsvd} and
choose the vectors
\begin{equation*}
  \vec a = \widetilde s_1^2 (I - W_k W_k^*) A^*\!A \vec{\widetilde x}_1
  \quad\text{and}\quad
  \vec b = \widetilde c_1^2 (I - W_k W_k^*) B^*B \vec{\widetilde x}_1
\end{equation*}
for expansion, and set
\begin{equation*}
  \vec r = \vec a - \vec b
  \quad\text{and}\quad
  \vec s = (\vec r^*\vec b) \vec a - (\vec r^*\vec a) \vec b,
\end{equation*}
since the residual norm $\|\vec r\|$ is required anyway.  Moreover, this
choice ensures at least the same improvement per iteration as the
residual expansion from generalized Davidson.  After the expansion and
extraction, a low-quality search direction may be removed.  Below we
describe the process in more detail.

In Section~\ref{tcgsvd:sec:gdgsvd} we have seen that
$A^*\!A\vec{\widetilde x}_1 = \widetilde c_1 A^*\vec{\widetilde u}_1$
and $B^*B\vec{\widetilde x}_1 = \widetilde s_1 B^*\vec{\widetilde v}_1$;
hence, suppose that $W_{k+2}$ is obtained by extending $W_k$ with the
$A^* \vec{\widetilde u}_1$ and $B^* \vec{\widetilde v}_1$ after
orthonormalization. Then we can compute the reduced QR-decompositions
\begin{equation}\label{tcgsvd:eq:mdqr}
  AW_{k+2} = U_{k+2} H_{k+2}
  \quad\text{and}\quad
  BW_{k+2} = V_{k+2} K_{k+2},
\end{equation}
and the triangular-form GSVD
\begin{equation*}
  \begin{split}
    H_{k+2}
    \begin{bmatrix}
      \widetilde W_{k+1} & \vec{\widetilde w}_{k+2}
    \end{bmatrix}
    &=
    \begin{bmatrix}
      \widetilde U_{k+1} & \vec{\widetilde u}_{k+2}
    \end{bmatrix}
    \begin{bmatrix}
      \widetilde C_{k+1} \\ & \widetilde c_{k+2}
    \end{bmatrix}
    \begin{bmatrix}
      \widetilde R_{k+1} & \vec{\widetilde r}_{k+1,k+2} \\
      & \widetilde r_{k+2,k+2}
    \end{bmatrix},
    \\
    K_{k+2}
    \begin{bmatrix}
      \widetilde W_{k+1} & \vec{\widetilde w}_{k+2}
    \end{bmatrix}
    &=
    \begin{bmatrix}
      \widetilde V_{k+1} & \vec{\widetilde v}_{k+2}
    \end{bmatrix}
    \begin{bmatrix}
      \widetilde S_{k+1} \\ & \widetilde s_{k+2}
    \end{bmatrix}
    \begin{bmatrix}
      \widetilde R_{k+1} & \vec{\widetilde r}_{k+1,k+2} \\
      & \widetilde r_{k+2,k+2}
    \end{bmatrix},
  \end{split}
\end{equation*}
where we may assume without loss of generality that $(\widetilde
c_{k+2}, \widetilde s_{k+2})$ is the generalized singular pair furthest
from the desired pair. By combining the partitioned decompositions above
with \eqref{tcgsvd:eq:mdqr}, we see that the objective becomes the
removal of $\linspan\{W_{k+2} \vec{\widetilde w}_{k+2}\}$ from the
search space.  One way to truncate this unwanted direction from the
search space, is to perform a restart conform
Section~\ref{tcgsvd:sec:gdgsvd} and compute
\begin{equation}\label{tcgsvd:eq:mdrestart}
  U_{k+2} \widetilde U_{k+1},
  \qquad
  V_{k+2} \widetilde V_{k+1},
  \qquad
  W_{k+2} \widetilde W_{k+1},
  \qquad
  \widetilde C_{k+1} \widetilde R_{k+1},
  \quad\text{and}\quad
  \widetilde S_{k+1} \widetilde R_{k+1}
\end{equation}
explicitly. However, with $\mathcal O(nk^2)$ floating-point
operations per iteration, the computational cost of this approach is too
high.  The key to a faster method is to realize that we only need to be
able to truncate
\begin{equation*}
  U_{k+2}\vec{\widetilde u}_{k+2},
  \qquad
  V_{k+2}\vec{\widetilde v}_{k+2},
  \qquad
  W_{k+2}\vec{\widetilde w}_{k+2},
  \qquad
  \widetilde c_{k+2},
  \quad\text{and}\quad
  \widetilde s_{k+2},
\end{equation*}
but do not require the matrices in \eqref{tcgsvd:eq:mdrestart}.  To this
end, let $P$, $Q$, and $Z$ be Householder reflections of the form
\begin{equation*}
  P = I - 2\frac{\vec p\vec p^*}{\vec p^*\vec p},
  \qquad
  Q = I - 2\frac{\vec q\vec q^*}{\vec q^*\vec q},
  \quad\text{and}\quad
  Z = I - 2\frac{\vec z\vec z^*}{\vec z^*\vec z},
\end{equation*}
with $\vec p$, $\vec q$, and $\vec z$ such that
\begin{equation*}
  P\vec e_{k+2} = \vec{\widetilde u}_{k+2},
  \qquad
  Q\vec e_{k+2} = \vec{\widetilde v}_{k+2},
  \qquad\text{and}\qquad
  Z\vec e_{k+2} = \vec{\widetilde w}_{k+2}.
\end{equation*}
%
Applying the Householder matrices yields
\begin{equation}\label{tcgsvd:eq:refl}
  A(W_{k+2} Z) = (U_{k+2} P) (P^* H_{k+2} Z)
  \quad\text{and}\quad
  B(W_{k+2} Z) = (V_{k+2} Q) (Q^* K_{k+2} Z),
\end{equation}
which can be computed in $\Landau O(nk)$ through rank-1 updates.  It is
straightforward to verify that the bottom rows of $P^* H_{k+2} Z$ and
$Q^* K_{k+2} Z$ are multiples of $\vec e_{k+2}^*$, e.g.,
\begin{equation*}
  \vec e_{k+2}^* P^* H_{k+2} Z
  = \vec{\widetilde u}_{k+2}^* (\widetilde U \widetilde C \widetilde R \widetilde W^*)Z
  = \widetilde c_{k+2} \widetilde r_{k+2,k+2} \vec{\widetilde w}_{k+2}^* Z
  = \widetilde c_{k+2} \widetilde r_{k+2,k+2} \vec e_{k+2}^*.
\end{equation*}
As a result, \eqref{tcgsvd:eq:refl} can be partitioned as
\begin{equation*}
  \begin{split}
    A \begin{bmatrix} W_{k+1} & W_{k+2} \vec{\widetilde w}_{k+2} \end{bmatrix}
    &=\begin{bmatrix} U_{k+1} & U_{k+2} \vec{\widetilde u}_{k+2} \end{bmatrix}
      \begin{bmatrix} H_{k+1} & \times \\& \widetilde c_{k+2} \widetilde
      r_{k+2,k+2} \end{bmatrix},
    \\
    B \begin{bmatrix} W_{k+1} & W_{k+2} \vec{\widetilde w}_{k+2} \end{bmatrix}
    &=\begin{bmatrix} V_{k+1} & V_{k+2} \vec{\widetilde v}_{k+2} \end{bmatrix}
      \begin{bmatrix} K_{k+1} & \times \\& \widetilde c_{k+2} \widetilde
      r_{k+2,k+2} \end{bmatrix},
    \end{split}
\end{equation*}
defining $U_{k+1}$, $V_{k+1}$, $W_{k+1}$, $H_{k+1}$, and $K_{k+1}$. This
partitioning can be truncated to obtain
\begin{equation}\label{tcgsvd:eq:trimqr}
  AW_{k+1} = U_{k+1} H_{k+1}
  \quad\text{and}\quad
  BW_{k+1} = V_{k+1} K_{k+1},
\end{equation}
where $U_{k+1}$, $V_{k+1}$, and $W_{k+1}$ have orthonormal columns, but
$H_{k+1}$ and $K_{k+1}$ are not necessarily upper-triangular.  The
algorithm is summarized below in Algorithm~\ref{tcgsvd:alg:mdgsvd}.


\begin{algorithm}[Multidirectional GSVD (MDGSVD)]\label{tcgsvd:alg:mdgsvd}
\strut\\*
  \textbf{Input:} Matrix pair $(A,B)$, starting vectors $\vec w_1$ and
    $\vec w_2$, minimum and maximum dimensions $j < \ell$. \\
  \textbf{Output:} $AW_{j} = U_{j} C_{j} R_{j}$ and $BW_{j} = V_{j}
    S_{j} R_{j}$ approximating a partial GSVD. \\
  \tab[\phantom01.] Set $W_0 = []$. \\
  \tab[\phantom02.] \textbf{for} \textit{number of restarts} 
    \textbf{and} \emph{not converged (cf., e.g.,~\eqref{tcgsvd:eq:cnvtol})}
    \textbf{do} \\
  \tab[\phantom03.]\tab  \textbf{for} $k = 0$, 1, \dots,
    $ \ell-2$ \textbf{do} \\
  \tab[\phantom04.]\tab\tab  Let $\vec w_{k+1} = (I - W_k
    W_k^*) \vec w_{k+1}$, and $\vec w_{k+1} = \vec
    w_{k+1} / \|\vec w_{k+1}\|$. \\
  \tab[\phantom05.]\tab\tab  Let $\vec w_{k+2} = (I - W_{k+1}
    W_{k+1}^*) \vec w_{k+2}$ and $\vec w_{k+2} = \vec
    w_{k+2} / \|\vec w_{k+2}\|$. \\
  \tab[\phantom06.]\tab\tab  Update the QR-decompositions \\
  \tab\tab\tab\tab $AW_{k+2} = U_{k+2} H_{k+2}$ and $BW_{k+2} = V_{k+2}
    K_{k+2}$.\\
  \tab[\phantom07.]\tab\tab  Compute the GSVD $H_{k+2} = \widetilde U
    \widetilde C \widetilde R \widetilde W^*$ and $K_{k+2} =
    \widetilde V \widetilde S \widetilde R \widetilde W^*$. \\
  \tab[\phantom08.]\tab\tab  Let $P$, $Q$, and $Z$ be Householder
    reflections such that\\
  \tab\tab\tab\tab
    $P\vec e_{k+2} = \vec{\widetilde u}_{k+2}$,
    $Q\vec e_{k+2} = \vec{\widetilde v}_{k+2}$, and
    $Z\vec e_{k+2} = \vec{\widetilde z}_{k+2}$. \\
  \tab[\phantom09.]\tab\tab   Let
    $U_{k+2} = U_{k+2} P$,
    $V_{k+2} = V_{k+2} Q$,
    $W_{k+2} = W_{k+2} Z$, \\
  \tab\tab\tab\tab
    $H_{k+2} = P^* H_{k+2} Z$, and
    $K_{k+2} = Q^* K_{k+2} Z$. \\
  \tab[10.]\tab\tab  $\vec w_{k+2} = A^*\vec{\widetilde u}_1$
    and $\vec w_{k+3} = B^*\vec{\widetilde v}_1$. \\
  \tab[11.]\tab\tab  \textbf{if} $j\le k$ \textbf{and}
    \textit{converged (cf., e.g.,~\eqref{tcgsvd:eq:cnvtol})}
    \textbf{then} \textbf{break} \\
  \tab[12.]\tab  \textbf{end} \\
  \tab[13.]\tab  Partition $\widetilde U$, $\widetilde V$,
    $\widetilde W$, $\widetilde C$, $\widetilde S$, and $\widetilde R$
    according to \eqref{tcgsvd:eq:part}. \\
  \tab[14.]\tab  Let $U_j = U_k \widetilde U_1$, $V_j =
    V_k \widetilde V_1$, and $W_j = W_k \widetilde W_1$. \\
  \tab[15.]\tab  Let $H_j = \widetilde C_{1} \widetilde R_{11}$
    and $K_j = \widetilde S_{1} \widetilde R_{11}$. \\
  \tab[16.]  \textbf{end}
\end{algorithm}


Algorithm~\ref{tcgsvd:alg:mdgsvd} is a simplified description for the
sake of clarity. For instance, the expansion vectors may be linearly
dependent in practice, and it may be desirable to expand a search space
of dimension $\ell-1$ with only the residual instead of two vectors.
Another missing feature that might be required in practice is deflation,
which is the topic of the next section.



\section{Deflation and the truncated GSVD}
\label{tcgsvd:sec:tgsvd}

Deflation is used in eigenvalue computations to prevent iterative
methods from recomputing known eigenpairs. Since
Algorithm~\ref{tcgsvd:alg:gdgsvd} and Algorithm~\ref{tcgsvd:alg:mdgsvd}
compute generalized singular values and vectors one at a time, deflation
may be necessary for applications where more than one generalized
singular pair is required. The truncated GSVD is an example of such an
application.  There are at least two ways in which generalized singular
values and vectors can be deflated, namely by transformation and by
restriction.  These two approaches have been inspired by their
counterparts for the symmetric eigenvalue problem (cf., e.g.,
\textcite[Ch.~5]{Par98}).  We only describe the two approaches for $m,p
\ge n$ to avoid clutter, but note that they can be adapted to the
general case.

The restriction approach is related to the truncation described in the
previous section and may be used to deflate a single generalized
singular pair at a time. Suppose we wish to deflate the simple pair
$(c_{1},s_{1})$ and let the GSVD of $(A,B)$ be partitioned as
\begin{equation*}
  \begin{split}
    A &= \begin{bmatrix} \vec u_1 & U_2 \end{bmatrix}
      \begin{bmatrix} c_{1} \\ & C_{2} \end{bmatrix}
      \begin{bmatrix} r_{11} & \vec r_{12}^* \\ & R_{22} \end{bmatrix}
      \begin{bmatrix} \vec w_1^* \\ W_2^* \end{bmatrix}, \\
    B &= \begin{bmatrix} \vec v_1 & V_2 \end{bmatrix}
      \begin{bmatrix} s_{1} \\ & S_{2} \end{bmatrix}
      \begin{bmatrix} r_{11} & \vec r_{12}^* \\ & R_{22} \end{bmatrix}
      \begin{bmatrix} \vec w_1^* \\ W_2^* \end{bmatrix},
  \end{split}
\end{equation*}
where $C_{2}$ and $S_{2}$ may be rectangular. Then, with Householder
reflections $P$, $Q$, and $Z$, satisfying
\begin{equation*}
  P\vec u_1 = \vec e_1,
  \qquad
  Q\vec v_1 = \vec e_1,
  \quad\text{and}\quad
  Z\vec w_1 = \vec e_1,
\end{equation*}
it holds that
\begin{equation*}
    PAZ = \begin{bmatrix} c_{1} r_{11} & \times \\
      & \widehat A \end{bmatrix}
    \quad\text{and}\quad
    QBZ = \begin{bmatrix} s_{1} r_{11} & \times \\
      & \widehat B \end{bmatrix},
\end{equation*}
defining $\widehat A$ and $\widehat B$.  At this point, the generalized
singular pairs of $(\widehat A, \widehat B)$ are the generalized
singular pairs of $(A,B)$ other than $(c_1,s_1)$.  Additional
generalized singular pairs can be deflated inductively.

An alternative that allows for the deflation of multiple generalized
singular pairs simultaneously is the restriction approach.
To derive this approach, let the GSVD of $(A,B)$ be partitioned as
\begin{equation}\label{tcgsvd:eq:ABsplit}
  \begin{split}
    A &= \begin{bmatrix} U_1 & U_2 \end{bmatrix}
      \begin{bmatrix} C_{1} \\ & C_{2} \end{bmatrix}
      \begin{bmatrix} R_{11} & R_{12} \\ & R_{22} \end{bmatrix}
      \begin{bmatrix} W_1^* \\ W_2^* \end{bmatrix}, \\
    B &= \begin{bmatrix} V_1 & V_2 \end{bmatrix}
      \begin{bmatrix} S_{1} \\ & S_{2} \end{bmatrix}
      \begin{bmatrix} R_{11} & R_{12} \\ & R_{22} \end{bmatrix}
      \begin{bmatrix} W_1^* \\ W_2^* \end{bmatrix},
  \end{split}
\end{equation}
where $C_{1}$ and $S_{1}$ are square and must be deflated, while $C_{2}$
and $S_{2}$ may be rectangular and must be retained.  Therefore, the
desired generalized singular pairs are deflated by working with the
operators
\begin{equation}\label{tcgsvd:eq:restrictAB}
  \begin{split}
    \widehat A
    &= U_2 C_{2} R_{22} W_2^*
    = U_2 U_2^* A W_2 W_2^*
    = (I - U_1 U_1^*) A (I - W_1 W_1^*), \\
    \widehat B
    &= V_2 S_{2} R_{22} W_2^*
    = V_2 V_2^* B W_2 W_2^*
    = (I - V_1 V_1^*) B (I - W_1 W_1^*),
  \end{split}
\end{equation}
restricted to $\subspace W_2 = \linspan \{ W_2 \}$.  An important
benefit of this approach is that the restriction may be performed
implicitly during the iterations.  For example, if
\eqref{tcgsvd:eq:part} is such that
\begin{equation*}
  U_k \widetilde U_1 = U_1,
  \quad
  V_k \widetilde V_1 = V_1,
  \quad
  W_k \widetilde W_1 = W_1,
  \quad
  \widetilde C_{1} = C_{1},
  \quad
  \widetilde S_{1} = S_{1},
  \quad\text{and}\quad
  \widetilde R_{11} = R_{11};
\end{equation*}
then
%
\begin{equation*}
  \widehat A W_k \widetilde W_2
  = U_k \widetilde U_2 \widetilde C_{2} \widetilde R_{22}
  \quad\text{and}\quad
  \widehat B W_k \widetilde W_2
  = V_k \widetilde V_2 \widetilde S_{2} \widetilde R_{22},
\end{equation*}
where the right-hand sides are available without explicitly working with
$\widehat A$ and $\widehat B$. In addition, if we define the
approximations for the next generalized singular
pair and corresponding vectors as
\begin{equation*}
  \begin{split}
    \alpha &= \vec e_1^* \widetilde C_{2} \vec e_1, &\quad
    \beta &= \vec e_1^* \widetilde S_{2} \vec e_1, &\quad
    \rho &= \vec e_1^* \widetilde R_{22} \vec e_1, \\
    \vec{\widetilde u} &= U_k \widetilde U_2 \vec e_1, &\quad
    \vec{\widetilde v} &= V_k \widetilde V_2 \vec e_1, &\quad
    \vec{\widetilde w} &= W_k \widetilde W_2 \vec e_1,
  \end{split}
\end{equation*}
cf.~Section~\ref{tcgsvd:sec:gdgsvd}, and
\begin{equation*}
  \vec{\widetilde x}
  = \rho^{-1} W_k \widetilde W
  \begin{bmatrix}
    \widetilde R_{11}^{-1} \widetilde R_{12} \vec e_1 \\
    \vec e_1
  \end{bmatrix}
  =
  \rho^{-1} (W_k \widetilde W_1 \widetilde R_{11}^{-1}
    \widetilde R_{12} \vec e_1 + \vec{\widetilde w});
\end{equation*}
then the residual
\begin{equation*}
  \begin{split}
    \vec r = \rho^{-1} (\beta^2 \widehat A^*\!\widehat A
    {} - \alpha^2 \widehat B^*\widehat B) \vec{\widetilde w}
    &= \alpha\beta (\beta \widehat A^* \vec{\widetilde u}
    {} - \alpha \widehat B^* \vec{\widetilde v}) \\
    &= \alpha\beta (\beta A^* \vec{\widetilde u}
    {} - \alpha B^* \vec{\widetilde v})
    = (\beta^2 A^*\!A - \alpha^2 B^*B) \vec{\widetilde x}
  \end{split}
\end{equation*}
and expansion vector(s) can also be computed without $\widehat A$ and
$\widehat B$.


It may be instructive to point out that the restriction approach for
deflation corresponds to a splitting method for general form Tikhonov
regularization described in \autocite{HR10genregu} and references. This
method separates the penalized part of the solution from the unpenalized
part associated with the nullspace of the regularization operator,
essentially deflating specific generalized singular values and vectors.
Consider, for instance, the minimization problem
\begin{equation*}
  \argmin_{\vec x} \|A\vec x - \vec b\|^2 + \mu \|B\vec x\|^2
\end{equation*}
for some $\mu > 0$. Assume for the sake of simplicity that $p \ge n$,
adding zero rows to $B$ if necessary, and suppose that $W_1$ is a basis
for the nullspace of $B$; then we obtain
\begin{equation}\label{tcgsvd:eq:tGSVDsplit}
  \begin{split}
    \|A\vec x - \vec b\|^2 + \mu \|B\vec x\|^2
    &= \|U_1U_1^*AW_1W_1^*\vec x - (U_1U_1^*\vec b - U_1U_1^*AW_2W_2^*\vec x)\|^2 \\
    &\quad + \|U_2U_2^*AW_2W_2^*\vec x - U_2U_2^*\vec b\|^2 + \mu
    \|V_2 V_2^* BW_2W_2^*\vec x\|^2
  \end{split}
\end{equation}
by following the splitting approach and using that $U_2 U_2^* A W_1
W_1^* \vec x = \vec 0$.  Furthermore, with $\vec y_1 = W_1^* \vec x$ and
$\vec y_2 = W_2^*\vec x$, the first part of the right-hand side of
\eqref{tcgsvd:eq:tGSVDsplit} reduces to
\begin{equation*}
    \|(U_1^*AW_1) \vec y_1 - (U_1^*\vec b - U_1^*AW_2\vec y_2)\|^2
    = \|R_{11} \vec y_1 - U_1^*(\vec b - AW_2\vec y_2)\|^2,
\end{equation*}
which vanishes for $\vec y_1 = R_{11}^{-1} U_1^*(\vec b - AW_2\vec
y_2)$.  The remaining part may be written as
\begin{equation*}
  \|\widehat A W_2\vec y_2 - U_2 U_2^* \vec b\|^2
  + \mu \|\widehat B W_2\vec y_2\|^2,
\end{equation*}
where we recognize the deflated matrices from
\eqref{tcgsvd:eq:restrictAB}. A similar expression can be derived for
deflation through restriction, but does not provide additional insight.



\section{Error analysis}
\label{tcgsvd:sec:err}

In this section we are concerned with the quality of the computed
approximations, and develop Rayleigh--Ritz theory that is useful for the
GSVD. In particular, we will generalize several known results for the
$n\times n$ standard Hermitian eigenvalue problem to the Hermitian
positive definite generalized eigenvalue problem
\begin{equation}\label{tcgsvd:eq:gep}
  N \vec x = \lambda M\vec x,
  \qquad
  M > 0,
  \qquad
  M = L^2,
\end{equation}
with eigenvalues $\lambda_1 \ge \lambda_2 \ge \cdots \ge \lambda_n$.
This generalized problem is applicable in our context with $N = A^*\!A$
and $M = X^{-*}X^{-1}$ if we are interested in the largest generalized
singular values, or with $N = B^*B$ and $M = X^{-*}X^{-1}$ if we are
interested in the smallest generalized singular values; and corresponds
to the standard problem
\begin{equation}\label{tcgsvd:eq:ep}
  L^{-1}NL^{-1} \vec y = \lambda \vec y,
  \qquad
  \vec y = L \vec x,
\end{equation}
with the same eigenvalues. Hence, if the subspace $\subspace W$ is a
search space for \eqref{tcgsvd:eq:gep}, then it is natural to consider
$\subspace Z = L\subspace W$ as a search space for \eqref{tcgsvd:eq:ep}
and to associate every approximate generalized eigenvector $\vec w \in
\subspace W$ with an approximate eigenvector $\vec z = L \vec w \in
\subspace Z$.  The corresponding Rayleigh quotients satisfy
\begin{equation}\label{tcgsvd:eq:theta}
  \theta
  = \frac{\vec w^*N\vec w}{\vec w^*M\vec w}
  = \frac{\vec z^*L^{-1}NL^{-1}\vec z}{\vec z^*\vec z}
\end{equation}
and define the approximate eigenvalue $\theta$.

Key to extending results for the generalized problem
\eqref{tcgsvd:eq:ep} to results for the standard problem
\eqref{tcgsvd:eq:gep}, is to introduce generalized sines, cosines, and
tangents, with respect to the $M$-norm defined by $\|\vec x\|_M^2 = \vec
x^*M\vec x = \|L\vec x\|^2$. Generalizations of these trigonometric
functions have previously been considered by \textcite{Spence06}, and
the generalized tangent can also be found in
\autocite[Thm.~15.9.3]{Par98}; however, we believe the treatment and
results presented here to be new.  The regular sine for two nonzero
vectors $\vec y$ and $\vec z$ can be defined as
\begin{equation*}
  \sin(\vec z, \vec y)
  = \frac{\big\|\big(I-\proj{\vec z}\big)\vec y\big\|}{\|\vec y\|}
  = \frac{\big\|\big(I-\proj{\vec y}\big)\vec z\big\|}{\|\vec z\|},
\end{equation*}
where it is easily verified that the above two expressions are equal
indeed.  Substituting $L\vec x$ for $\vec y$ and $L\vec w$ for $\vec z$
yields the $M$-sine defined by
\begin{equation*}
  \sin_M(\vec w, \vec x)
  = \sin(L\vec w, L\vec x)
  = \frac{\big\|\big(I - \projM{\vec w}\big)\vec x\big\|_M}{\|\vec x\|_M}
  = \frac{\big\|\big(I - \projM{\vec x}\big)\vec w\big\|_M}{\|\vec w\|_M}.
\end{equation*}
Again, it may be checked that the above two expressions are equal.  The
regular cosine is given by
\begin{equation*}
  \cos(\vec z, \vec y)
  = \frac{\big\|\proj{\vec z}\vec y\big\|}{\|\vec y\|}
  = \frac{\big\|\proj{\vec y}\vec z\big\|}{\|\vec z\|}
  = \frac{|\vec z^*\vec y|}{\|\vec z\|\,\|\vec y\|},
\end{equation*}
and with the same substitution we find the $M$-cosine
\begin{equation*}
  \cos_M(\vec w, \vec x)
  = \cos(L\vec w, L\vec x)
  = \frac{\big\|\projM{\vec w}\vec x\big\|_M}{\|\vec x\|_M}
  = \frac{\big\|\projM{\vec x}\vec w\big\|_M}{\|\vec w\|_M}
  = \frac{|\vec w^*M\vec x|}{\|\vec w\|_M \|\vec x\|_M}.
\end{equation*}
The $M$-tangent is now naturally defined as $\tan_M(\vec w, \vec x) =
\sin_M(\vec w, \vec x) / \cos_M(\vec w, \vec x)$. We can derive the
$M$-sines, $M$-cosines, and $M$-tangents between subspaces and vectors
with a similar approach. For instance, let $W$ and $LW$ denote bases for
$\subspace W$ and $\subspace Z$, respectively; then
\begin{equation*}
  \sin(\subspace Z,\vec y)
  = \frac{\big\|\big(I - \proj{\vec y}\big)
    Z(Z^*Z)^{-1}Z^*\vec y\big\|}{\|Z(Z^*Z)^{-1}Z^*\vec y\|}
  \quad\text{and}\quad
  \cos(\subspace Z, \vec y)
  = \frac{\|Z(Z^*Z)^{-1}Z^*\vec y\|}{\|\vec y\|},
\end{equation*}
so that
\begin{align*}
  \sin_M(\subspace W, \vec x)
  &= \sin(L\subspace W, L\vec x)
  = \frac{\big\|\big(I - \projM{\vec x}\big)
    W(W^*MW)^{-1}W^*M\vec x\big\|_M}{\|W(W^*MW)^{-1}W^*M\vec x\|_M},
  \\
  \cos_M(\subspace W, \vec x)
  &= \cos(L\subspace W, L\vec x)
  = \frac{\|W(W^*MW)^{-1}W^*M\vec x\|_M}{\|\vec x\|_M},
\end{align*}
and $\tan_M(\subspace W, \vec x) = \sin_M(\subspace W, \vec x) /
\cos_M(\subspace W, \vec x)$.  It is important to note that  $\sin_M$,
$\cos_M$, and $\tan_M$ can all be computed without the matrix square
root $L$ of $M$.

Since our $M$-sines, $M$-cosines, and $M$-tangents equal their regular
counterparts, the extension of several known results for the standard
problem \eqref{tcgsvd:eq:ep} to results for the generalized problem
\eqref{tcgsvd:eq:gep} is immediate. Below is a selection of error
bounds, where we assume that the largest generalized eigenpair
$(\lambda_1,\vec x_1)$ is simple and is approximated by the Ritz pair
$(\theta_1, \vec w_1)$ of \eqref{tcgsvd:eq:gep} with respect to the
search space $\subspace W$.


\begin{proposition}[Generalization of, e.g.,
  {\autocite[Lemma~11.9.2]{Par98}}]
\label{tcgsvd:thm:sinw_lt_dtd2}
  \begin{equation*}
    \sin_M^2(\vec w_1, \vec x_1)
    \le \frac{\lambda_1 - \theta_1}{\lambda_1 - \lambda_2}.
  \end{equation*}
\end{proposition}



\begin{proposition}[Generalization of {\autocite[Thm.~2.1]{SE02bnds}}]
\label{tcgsvd:thm:dt_lt_dn_sinW}
  \begin{equation*}
    \lambda_1 - \theta_1
    \le (\lambda_1 - \lambda_n) \, \sin_M^2(\subspace W, \vec x_1).
  \end{equation*}
\end{proposition}


\noindent
The two propositions imply that $\vec w_1 \to \vec x_1$ when $\theta_1
\to \lambda_1$, with $\theta_1$ tending to $\lambda_1$ when
$\sin(\subspace W,\vec x_1) \to 0$. The next corollary is a
straightforward consequence.


\begin{corollary}[Generalization of {\autocite[Thm.~2.1]{SE02bnds}}]
\label{tcgsvd:thm:sinw_lt_dnd2_sinW}
  \begin{equation*}
    \sin_M^2(\vec w_1, \vec x_1)
    \le \frac{\lambda_1 - \lambda_n}{\lambda_1 - \lambda_2}\,
      \sin_M^2(\subspace W, \vec x_1)
    = \left(
      1 + \frac{\lambda_2 - \lambda_n}{\lambda_1 - \lambda_2}
    \right) \sin_M^2(\subspace W, \vec x_1).
  \end{equation*}
\end{corollary}


\noindent
As a result of Corollary~\ref{tcgsvd:thm:sinw_lt_dnd2_sinW}, we can
expect $\sin_M(\vec w_1,\vec x_1)$ to be close to $\sin_M(\subspace W,
\vec x_1)$ if the eigenvalue $\lambda_1$ is well separated from the rest
of the spectrum. A sharper bound can be obtained by generalizing the
optimal bound from \textcite{SE02bnds}.


\begin{proposition}[Generalization of {\autocite[Thm.~3.2]{SE02bnds}}]
\label{tcgsvd:thm:sharp}
  Let $(\theta_j,\vec w_j)$ denote the Ritz pairs of the generalized
  problem \eqref{tcgsvd:eq:gep} with respect to $\subspace W$, and
  define
  \begin{equation*}
    \delta_{\subspace W} = \min\,\sin_M(\vec w_j, \vec x_1)
  \end{equation*}
  as the smallest of all $M$-sines between the Ritz vectors $\vec w_j$
  and the generalized eigenvector $\vec x_1$. Furthermore, define for
  any $\epsilon > 0$ the maximum
  \begin{equation*}
    \delta_k(\epsilon) = \max_{\subspace W}\, \{
      \delta_{\subspace W} \mid \dim(\subspace W) = k,
      \sin_M(\subspace W, \vec x_1) \le \epsilon
    \}.
  \end{equation*}
  If $(\theta_{\subspace W},\vec w_{\subspace W})$ is the Ritz pair for
  which $\delta_{\subspace W}$ is realized and $0 \le \epsilon <
  (\lambda_1 - \lambda_2) / (\lambda_1 - \lambda_n)$, then
  $\theta_{\subspace W} = \theta_1 > \lambda_2$ and
  \begin{equation*}
    \delta_k^2(\epsilon) = \frac{1}{2}(1 + \epsilon^2)
    {} - \frac{1}{2} \sqrt{(1 - \epsilon^2)^2 - \kappa\epsilon^2}
    \quad\text{with}\quad
    \kappa = \frac{(\lambda_2 - \lambda_n)^2}{(\lambda_1 -
    \lambda_n)(\lambda_1 - \lambda_2)},
  \end{equation*}
  for all $k \in \{2, \dots, n-1\}$.
\end{proposition}


\noindent
The quantity $\delta_k^2(\epsilon)$ is not particularly elegant, but is
sharp and can be used to obtain the following upper bound, which is
sharper than the bound in Corollary~\ref{tcgsvd:thm:sinw_lt_dnd2_sinW}.


\begin{corollary}[Generalization of {\autocite[Cor.~3.3]{SE02bnds}}]
  If the conditions in Proposition~\ref{tcgsvd:thm:sharp} are satisfied,
  then
  \begin{equation*}
    \sin_M^2(\vec w_1, \vec x_1)
    \le \sin_M^2(\subspace W, \vec x_1)
    {} + \frac{\kappa}{2} \tan_M^2(\subspace W, \vec x_1).
  \end{equation*}
\end{corollary}


Now that we have extended a number of results for the standard problem
\eqref{tcgsvd:eq:ep} to the generalized problem \eqref{tcgsvd:eq:gep},
it may be worthwhile to bound the generalized sine $\sin_M$ in terms
of the standard sine.


\begin{proposition}
  Let $\kappa = \kappa(M)$ be the condition number of $M$, then
  \[
    \frac{1}{\kappa}  \sin^2(\vec w, \vec x)
    \le \sin_M^2(\vec w, \vec x)
    \le \frac{1}{4} (\kappa+1)^2 \sin^2(\vec w, \vec x).
  \]
\end{proposition}
\begin{proof}
  Without loss of generality we assume $\|\vec w\| = \|\vec x\| =
  1$, so that
  \[
    \lambda_{\min}(M) \le \|\vec x\|_M^2 \le \lambda_{\max}(M).
  \]
  The first inequality follows from
  \begin{equation*}
    \begin{split}
      \sin^2(\vec w, \vec x)
      &= \big\|(I - \vec w \vec w^*)
        \big(I - \projM{\vec w}\big)\vec x\big\|^2 \\
      &\le \|\vec x\|_M^2 \, \frac{\big\|\big(I - \projM{\vec w}\big)\vec x\big\|^2}{%
        \big\|\big(I - \projM{\vec w}\big)\vec x\big\|_M^2} \, \sin_M^2(\vec w, \vec x)
      \le \frac{\lambda_{\max}(M)}{\lambda_{\min}(M)} \, \sin_M^2(\vec w, \vec x).
    \end{split}
  \end{equation*}
  For the second inequality, it follows from, e.g., \autocite{DS06proj}
  that
  \[
    \left\|I - \frac{\vec w\vec w^*M}{\vec w^*M\vec w}\right\|
    = \left\|\frac{\vec w\vec w^*M}{\vec w^*M\vec w}\right\|
    = \frac{\|M\vec w\|}{\vec w^*M\vec w}
    = \cos^{-1}(\vec w, M \vec w) \le \mu^{-1},
  \]
  where $\mu^{-1}$ is the inverse of the first anti-eigenvalue
  \autocite[Ch.~3.6]{GR97nr}
  \[
    \mu = \min_{\|\vec w\|=1} \frac{\vec w^*M\vec w}{\|M\vec w\|}.
  \]
  By applying Kantorovich' inequality we find \autocite[p.~68]{GR97nr}
  \[
    \mu^{-1} = \frac{1}{2} \, \frac{\lambda_{\min}(M) + \lambda_{\max}(M)}
    {\sqrt{\lambda_{\min}(M) \, \lambda_{\max}(M)}}
    = \frac{1}{2}\, \frac{\kappa+1}{\sqrt{\kappa}}.
  \]
  Finally, by combining the above and using
  \[
    \left(I - \frac{\vec w\vec w^*M}{\vec w^*M\vec w}\right)
    = \left(I - \frac{\vec w\vec w^*M}{\vec w^*M\vec w}\right)
    (I - \vec w\vec w^*)\vec x,
  \]
  we see that
  \begin{equation*}
    \begin{split}
      \sin_M^2(\vec w, \vec x)
      & = \frac{\big\|\big(I - \projM{\vec x}\big)\vec x\big\|_M^2}{\|\vec x\|_M^2}
      \le \frac{\lambda_{\max}(M)}{\lambda_{\min}(M)} \,
      \big\|\big(I - \projM{\vec w}\big)\vec x\big\|^2 \\
      &\le \kappa \, \big\|I - \projM{\vec w}\big\|^2 \, \|(I - \vec w\vec w^*)\vec x\|^2
      \le \frac{1}{4} (\kappa+1)^2 \sin^2(\vec w, \vec x),
    \end{split}
  \end{equation*}
  which concludes the proof.
\end{proof}


An interesting observation about $\sin_M$ in the context of the GSVD is
that $\|\vec f\|$ from Proposition~\ref{tcgsvd:thm:asymconv} equals
$\sin_M(\vec{\widetilde x}_1,\vec x_1)$ if $M = A^*\!A + B^*B =
X^{-*}X^{-1}$. Furthermore, it has been shown in the proof of
Proposition~\ref{tcgsvd:thm:asymconv} that the error in $\widetilde
c_1^2 = \|A\vec{\widetilde x}_1\|^2$ and $\widetilde s_1^2 =
\|B\vec{\widetilde x}_1\|^2$ is quadratic in $\|\vec f\|$. An
alternative is to express the approximation error in terms of the
residual. We have, for example, the following straightforward
Bauer--Fike-type result.


\begin{proposition}[Bauer--Fike for the GSVD]
\label{tcgsvd:thm:bf}
  Let $(\widetilde c, \widetilde s)$ be an approximate generalized
  singular pair with corresponding generalized singular vector
  $\vec{\widetilde x}$ and residual
  \begin{equation*}
    \vec r = (\widetilde s^2 A^*\!A - \widetilde c^2 B^*B)
    \vec{\widetilde x};
  \end{equation*}
  then there exists a generalized singular pair $(c_\star, s_\star)$ of
  $(A,B)$ such that
  \begin{equation*}
    |\widetilde s^2 c_\star^2 - \widetilde c^2 s_\star^2|
    \le \|X\|^2 \frac{\|\vec r\|}{\|\vec{\widetilde x}\|}.
  \end{equation*}
\end{proposition}
\begin{proof}
  The result follows from
  \begin{equation*}
    \begin{split}
      \frac{\|\vec r\|}{\|\vec{\widetilde x}\|}
      &\ge \sigma_{\min} (\widetilde s^2 A^*\!A - \widetilde c^2 B^*B) \\
      &= \sigma_{\min} (X^{-*}(\widetilde s^2 \Sigma_A^T\Sigma_A
      {} - \widetilde c^2 \Sigma_B^T\Sigma_B)X^{-1})
      \ge \sigma_{\min}^2 (X^{-1}) \min_{j}|\widetilde s^2 c_j^2
      {} - \widetilde c^2 s_j^2|.
    \end{split}
  \end{equation*}
\end{proof}


\noindent
An additional interesting observation is that if $\widetilde c$ and
$\widetilde s$ are scaled such that $\widetilde c^2 + \widetilde s^2 =
c_\star^2 + s_\star^2 = 1$, and the generalized singular values are
given by $\widetilde \sigma = \widetilde c / \widetilde s$ and
$\sigma_\star = c_\star / s_\star$; then
\begin{equation*}
    |\widetilde s^2 c_\star^2 - \widetilde c^2 s_\star^2|
    = |\widetilde s^2 - s_\star^2|
    = |c_\star^2 - \widetilde c^2|
    = \left|\frac{\widetilde \sigma^2}{1+\widetilde \sigma^2}
    {} - \frac{\sigma_\star^2}{1+\sigma_\star^2}\right|,
\end{equation*}
with the conventions $\infty / \infty = 1$ and $\infty - \infty = 0$.


The bound in Proposition~\ref{tcgsvd:thm:bf} may be rather pessimistic,
and we expect asymptotic convergence of order $\|\vec r\|^2$ due to the
relation with the symmetric eigenvalue problem.
It turns out that the desired result is easily generalized using the
$M$-sine and the $M^{-1}$-norm. Specifically, let $\theta$ be defined as
in \eqref{tcgsvd:eq:theta} and define the residual norms
\begin{equation*}
  \rho(\vec z)
  = \|(L^{-1}NL^{-1} - \theta I) \vec z\|
  \quad\text{and}\quad
  \rho_M(\vec w)
  = \rho(L\vec w)
  = \|(N - \theta M) \vec w\|_{M^{-1}};
\end{equation*}
then we can immeadiately derive the following proposition.

\begin{proposition}[Generalization of, e.g.,
  {\autocite[Thm.~11.7.1, Cor.~11.7.1]{Par98}}]
  \label{tcgsvd:thm:parlett}
  Suppose $\lambda_1 - \theta_1 < \theta_1 - \lambda_2$; then
  \begin{equation*}
    \frac{\rho_M(\vec w_1)}{\lambda_1 - \lambda_n}
    \le \sin_M(\vec w_1,\vec x_1)
    \le \tan_M(\vec w_1,\vec x_1)
    \le \frac{\rho_M(\vec w_1)}{\theta_1 - \lambda_2}
  \end{equation*}
  and
  \begin{equation*}
    \frac{\rho_M^2(\vec w_1)}{\lambda_1 - \lambda_n}
    \le \lambda_1 - \theta_1
    \le \frac{\rho_M^2(\vec w)}{\theta_1 - \lambda_2}.
  \end{equation*}
\end{proposition}

\noindent
Having the $M^{-1}$-norm for the residual instead of the $M$-norm might
be surprising; however, the former is a natural choice in this context;
see, e.g., \autocite[Ch.~15]{Par98}. Moreover,
Proposition~\ref{tcgsvd:thm:parlett} combined with the norm equivalence
\begin{equation*}
  \sigma_{\max}^{-1}(M)\, \|\vec r\|^2
  \le \|\vec r\|_{M^{-1}}^2
  \le \sigma_{\min}^{-1}(M)\, \|\vec r\|^2
\end{equation*}
implies that the converence of the generalized singular values must be
of order $\|\vec r\|^2$. This result is verified in an example in the
next section.



\section{Numerical experiments}
\label{tcgsvd:sec:num}

In this section we compare our new algorithms to JDGSVD and Zha's
modified Lanczos algorithm by using tests similar to the examples found
in \autocite{MH09jdgsvd} and \textcite{HZ96gsvd}. Additionally, we will
apply Algorithm~\ref{tcgsvd:alg:gdgsvd} and
Algorithm~\ref{tcgsvd:alg:mdgsvd} to general form Tikhonov
regularization by approximating truncated GSVDs for several test
problems. The first set of examples is detailed below.


\begin{example}\label{tcgsvd:ex:diag}
  Let $A = CD$ and $B = SD$ be two $n\times n$ matrices, where
  \begin{equation*}
    \begin{split}
      C &= \diag(c_j), & c_j &= (n - j + 1) / (2n), & S &= \sqrt{I - C^2}, \\
      D &= \diag(d_j), & d_j &= \lceil j / (n/4) \rceil + r_j,
    \end{split}
  \end{equation*}
  with $r_j$ drawn from the standard uniform distribution on the open
  interval $(0,1)$.
\end{example}



\begin{example}\label{tcgsvd:ex:orth}
  Let $C$ and $S$ be the same as in Example~\ref{tcgsvd:ex:diag}. Furthermore,
  let $A = UC\widetilde DW^*$ and $B=VS\widetilde DW^*$, where $U$, $V$,
  and $W$ are random orthonormal matrices, and $\widetilde D =
  \diag(\widetilde d_j)$ with
  \begin{equation*}
    \widetilde d_j = d_j - \min_{1\le j\le n} d_j + 10^{-\kappa}.
  \end{equation*}
  Three values for $\kappa$ are considered, (a) $\kappa = 6$, (b)
  $\kappa = 9$, and (c) $\kappa = 12$.
\end{example}



\begin{example}\label{tcgsvd:ex:hh}
  Let $C$ and $S$ be the same as in Example~\ref{tcgsvd:ex:diag}, and
  let $\widetilde D$ be the same as in Example~\ref{tcgsvd:ex:orth}. Let
  $\vec f$, $\vec g$, and $\vec h$ be random vectors on the unit
  $(n-1)$-sphere, and set
  \begin{equation*}
    A = (I - 2\vec f\vec f^*) C\widetilde D (I - 2\vec h\vec h^*)
    \quad\text{and}\quad
    B = (I - 2\vec g\vec g^*) S\widetilde D (I - 2\vec h\vec h^*).
  \end{equation*}
  Note that $I - 2\vec f \vec f^*$, $I - 2\vec g \vec g^*$, and $I -
  2\vec h \vec h^*$ are Householder reflections.
\end{example}






\begin{example}\label{tcgsvd:ex:rand}
  Let
  \begin{equation*}
    A = \texttt{sprand(n, n, 1e-1, 1)}
    \quad\text{and}\quad
    B = \texttt{sprand(n, n, 1e-1, 1e-2)},
  \end{equation*}
  where \texttt{sprand} is the MATLAB function with the same name.
\end{example}



\begin{table}[!htbp]
  \footnotesize
  \centerfloat
  \caption{The median number of matrix-vector products the algorithms
    require for Examples~\ref{tcgsvd:ex:diag}--\ref{tcgsvd:ex:rand} to
    compute an approximation satisfying \eqref{tcgsvd:eq:numtol}. The
    tolerance $\tau = 10^{-3}$ was used for Zha's modified Lanczos
    algorithm, while $\tau = 10^{-6}$ was used for the remaining
    algorithms.  The symbol $-$ indicates a failure to converge up to
    the desired tolerance within the maximum number of iterations
    specified in the text, and the column \textbf{Cond} contains the
    condition numbers of $[A^T\; B^T]^T$.}
  \label{tcgsvd:tab:synth}
  \begin{tabular}{@{}lr|rr|rr|rr|rr@{}}
    \toprule
    \textbf{Alg} &
    & \multicolumn{2}{c|}{\textbf{Zha}}
    & \multicolumn{2}{c|}{\textbf{JDGSVD}}
    & \multicolumn{2}{c|}{\textbf{GDGSVD}}
    & \multicolumn{2}{c}{\textbf{MDGSVD}} \\
    \textbf{Ex} & \textbf{Cond}
    & $\bm{\sigma_{\max}}$ & $\bm{\sigma_{\min}}$
    & $\bm{\sigma_{\max}}$ & $\bm{\sigma_{\min}}$
    & $\bm{\sigma_{\max}}$ & $\bm{\sigma_{\min}}$
    & $\bm{\sigma_{\max}}$ & $\bm{\sigma_{\min}}$ \\
    \midrule
    \ref{tcgsvd:ex:diag}  & 4.97e+00 &  3390 &  $-$ & 1524 & 6188 &  580 & 3072 &  502 & 730 \\
    \ref{tcgsvd:ex:orth}a & 3.99e+06 & 19082 &  $-$ & 2008 & 5396 &  992 & 2326 & 1054 & 622 \\
    \ref{tcgsvd:ex:orth}b & 3.99e+09 & 19082 &  $-$ & 2008 & 5396 &  998 & 2312 & 1036 & 628 \\
    \ref{tcgsvd:ex:orth}c & 3.99e+12 & 19082 &  $-$ & 2008 & 5374 &  998 & 2312 & 1030 & 622 \\
    \ref{tcgsvd:ex:hh}a   & 3.99e+06 & 17810 &  $-$ & 1964 & 5418 &  996 & 2318 & 1048 & 616 \\
    \ref{tcgsvd:ex:hh}b   & 3.99e+09 & 17810 &  $-$ & 1964 & 5418 &  996 & 2288 & 1036 & 616 \\
    \ref{tcgsvd:ex:hh}c   & 3.99e+12 & 17810 &  $-$ & 1964 & 5418 &  996 & 2288 & 1048 & 628 \\
    \ref{tcgsvd:ex:rand}  & 1.41e+00 &   $-$ & 1262 &  $-$ &  $-$ & 2334 &  244 & 2314 & 240 \\
    \bottomrule
  \end{tabular}
\end{table}



We generate the matrices from
Examples~\ref{tcgsvd:ex:diag}--\ref{tcgsvd:ex:rand} for $n = 1000$,
allowing us to verify the results.  For
Algorithm~\ref{tcgsvd:alg:gdgsvd} and Algorithm~\ref{tcgsvd:alg:mdgsvd}
we set the minimum dimension to 10, the maximum dimension to 30, and the
maximum number of restarts to 100. For JDGSVD we use the same minimum
and maximum dimensions in combination with a maximum of 10 and 1000
inner and outer iterations, respectively. Furthermore, we let JDGSVD use
standard extraction to find the largest generalized singular value, and
refined extraction to find the smallest generalized singular value.  We
have implemented Zha's modified Lanczos algorithm with LSQR, and let
LSQR use the tolerance $10^{-12}$ and a maximum of $\lceil 10 \sqrt{n}
\rceil = 320$ iterations. The maximum number of outer-iterations for the
modified Lanczos algorithm is 100.


\begin{figure}[htbp!]
  \centerfloat
  \tikzexternaldisable
  \begin{tikzpicture}
    \begin{semilogyaxis}[
          height=0.33\textwidth
        , width=0.44\textwidth
        , xlabel={\#MVs}
        , ylabel={$|\widetilde s^2 c_{\max}^2 - \widetilde c^2 s_{\max}^2|$}
        , every axis legend/.append style = {
              anchor = north east
            , at = {(0.98, 0.98)}
          }
        , enlarge x limits = false
        ]

        \addlegendentry{\ref{tcgsvd:ex:diag}}
        \addplot+ table [x index=0,y index=1] {data.hist.ex1.max.txt};
        \addlegendentry{\ref{tcgsvd:ex:orth}b}
        \addplot+ table [x index=0,y index=1] {data.hist.ex2b.max.txt};
        \addlegendentry{\ref{tcgsvd:ex:hh}b}
        \addplot+ table [x index=0,y index=1] {data.hist.ex3b.max.txt};
        \addlegendentry{\ref{tcgsvd:ex:rand}}
        \addplot+ table [x index=0,y index=1] {data.hist.ex5.max.txt};
    \end{semilogyaxis}
  \end{tikzpicture}
  \begin{tikzpicture}
    \begin{semilogyaxis}[
          height=0.33\textwidth
        , width=0.44\textwidth
        , xlabel={\#MVs}
        , ylabel={$|\widetilde s^2 c_{\min}^2 - \widetilde c^2 s_{\min}^2|$}
        , every axis legend/.append style = {
              anchor = north east
            , at = {(0.98, 0.98)}
          }
        , enlarge x limits = false
        , yticklabel pos=right
        ]
        \addlegendentry{\ref{tcgsvd:ex:diag}}
        \addplot+ table [x index=0,y index=1] {data.hist.ex1.min.txt};
        \addlegendentry{\ref{tcgsvd:ex:orth}b}
        \addplot+ table [x index=0,y index=1] {data.hist.ex2b.min.txt};
        \addlegendentry{\ref{tcgsvd:ex:hh}b}
        \addplot+ table [x index=0,y index=1] {data.hist.ex3b.min.txt};
        \addlegendentry{\ref{tcgsvd:ex:rand}}
        \addplot+ table [x index=0,y index=1] {data.hist.ex5.min.txt};
    \end{semilogyaxis}
  \end{tikzpicture}
  \tikzexternalenable
  \caption{The convergence history of MDGSVD as the errors from
    \eqref{tcgsvd:eq:numtol} compared to the number of matrix-vector
    products,  with results for the largest (left) and smallest (right)
    generalized singular pairs.}
  \label{tcgsvd:fig:conv}
\end{figure}
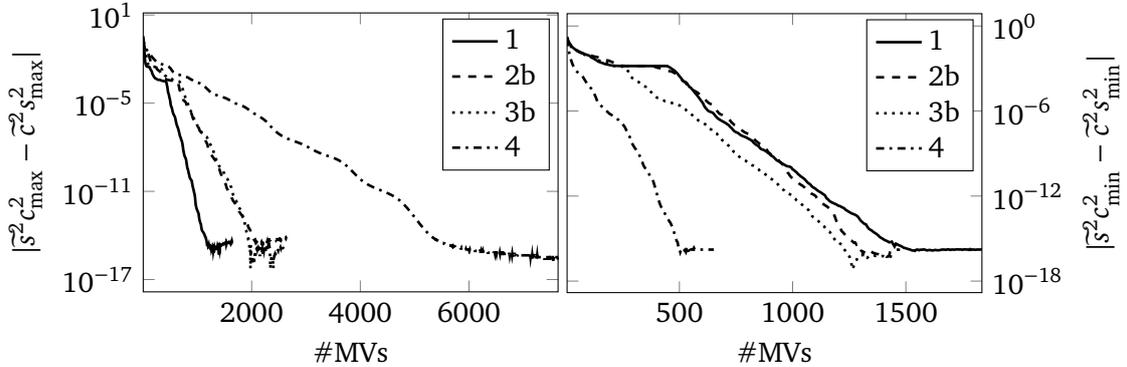


We run each test with 500 different starting vectors, and record the
number of matrix-vector products required until an approximate
generalized singular pair $(\widetilde c, \widetilde s)$ satisfies
\begin{equation}\label{tcgsvd:eq:numtol}
  |\widetilde s^2 c_{\max}^2 - \widetilde c^2 s_{\max}^2| < \tau
  \quad\text{or}\quad
  |\widetilde s^2 c_{\min}^2 - \widetilde c^2 s_{\min}^2| < \tau,
\end{equation}
where we use $\tau = 10^{-3}$ for Zha's modified Lanczos algorithm and
$\tau = 10^{-6}$ for the remaining algorithms. The median results are
shown in Table~\ref{tcgsvd:tab:synth}.  We notice that the convergence
of Zha's method is markedly slower here than in \cite{HZ96gsvd}.
Additional testing has indicated that the difference is caused by the
larger choice of $n$, which in turn decreases the gap between the
generalized singular pairs.  JDGSVD does not require accurate solutions
from the inner iterations and is significantly faster, but fails to
converge to a sufficiently accurate solution in the last example.
Compared to JDGSVD, GDGSVD approximately reduces the number of
matrix-vector multiplications by a factor of 2 for $\sigma_{\max}$ and
by a factor of $2$ to $2.4$ for $\sigma_{\min}$, and has no problem
finding a solution for the last example.  MDGSVD performs only slightly
worse than GDGSVD for the largest generalized singular pairs on
average, but uses approximately $4$ times fewer MVs than GDGSVD for the
smallest generalized singular pairs in almost all tests.


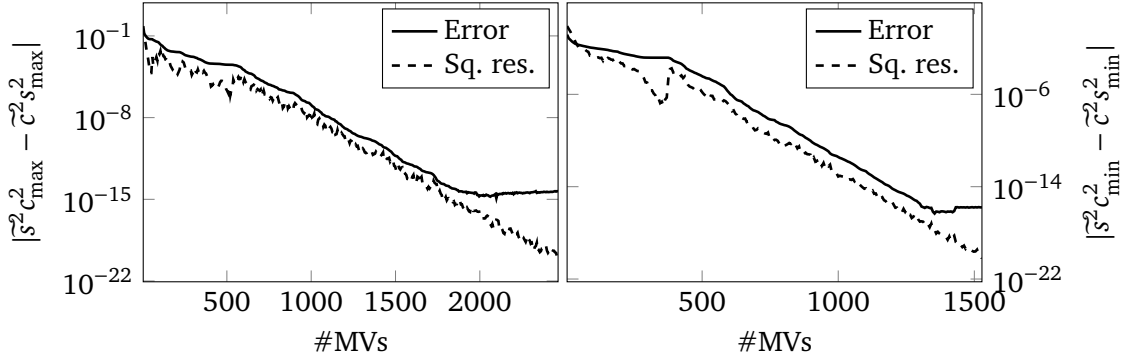
\begin{figure}[htbp!]
  \centerfloat
  \tikzexternaldisable
  \begin{tikzpicture}
    \begin{semilogyaxis}[
          height=0.33\textwidth
        , width=0.44\textwidth
        , xlabel={\#MVs}
        , ylabel={$|\widetilde s^2 c_{\max}^2 - \widetilde c^2 s_{\max}^2|$}
        , every axis legend/.append style = {
              anchor = north east
            , at = {(0.98, 0.98)}
          }
        , enlarge x limits = false
        ]

        \addlegendentry{Error}
        \addplot+ table [x index=0,y index=1] {data.res.ex2a.max.txt};
        \addlegendentry{Sq. res.}
        \addplot+ table [x index=0,y index=2] {data.res.ex2a.max.txt};
    \end{semilogyaxis}
  \end{tikzpicture}
  \begin{tikzpicture}
    \begin{semilogyaxis}[
          height=0.33\textwidth
        , width=0.44\textwidth
        , xlabel={\#MVs}
        , ylabel={$|\widetilde s^2 c_{\min}^2 - \widetilde c^2 s_{\min}^2|$}
        , every axis legend/.append style = {
              anchor = north east
            , at = {(0.98, 0.98)}
          }
        , enlarge x limits = false
        , yticklabel pos=right
        ]
        \addlegendentry{Error}
        \addplot+ table [x index=0,y index=1] {data.res.ex2a.min.txt};
        \addlegendentry{Sq. res.}
        \addplot+ table [x index=0,y index=2] {data.res.ex2a.min.txt};
    \end{semilogyaxis}
  \end{tikzpicture}
  \tikzexternalenable
  \caption{The errors of the largest (left) and smallest (right)
    generalized singular pairs approximations compared to the square of
    the relative residual norm in the right-hand side of
    \eqref{tcgsvd:eq:cnvtol}.  The results are for
    Example~\ref{tcgsvd:ex:orth}a and MDGSVD.}
  \label{tcgsvd:fig:error}
\end{figure}


Figure~\ref{tcgsvd:fig:conv} shows the convergence of MDGSVD.  The
monotone behavior and asymptotic linear convergence of the method are
clearly visible. We can also see that the asymptotic convergence is
significantly better than the worst-case bound from
Proposition~\ref{tcgsvd:thm:asymconv}. Figure~\ref{tcgsvd:fig:error}
shows a comparison between the relative residual norm
\eqref{tcgsvd:eq:cnvtol} and the convergence of the generalized singular
pairs for Example~\ref{tcgsvd:ex:orth}a. The results for the other
examples are similar, and are therefore omitted. Although the graphs
belonging to the smallest generalized singular pairs suggest temporary
misconvergence, the comparison still demonstrates that
\eqref{tcgsvd:eq:cnvtol} is an asymptotically suitable indicator for the
convergence of the generalized singular pairs. Moreover, the convergence
of the generalized singular pairs appears to be quadratic in the
residual norm.


\begin{example}\label{tcgsvd:ex:tgsvd}
  Given a large, sparse, and ill-conditioned matrix $A$, consider the
  problem of reconstructing exact data $\vec{x_\star}$ from measured
  data $\vec b = A\vec{x_\star} + \vec e$, where $\vec e$ is a noise
  vector. A regularized solution may be determined with general form
  Tikhonov regularization by computing
  \begin{equation*}
    \vec{x_\mu} = \argmin_{\vec x} \|A\vec x - \vec b\|^2 + \mu \|B\vec x\|^2
  \end{equation*}
  for some operator $B$ with $\Null(A) \cap \Null(B) = \{ \vec 0 \}$,
  and some parameter $\mu > 0$.  For the purpose of this example, we
  take several $n\times n$ matrices $A$ and length $n$ solution vectors
  $\vec{x_\star}$ from Regularization Tools \autocite{Regutools}, and
  for $B$ we use the $(n-1)\times n$ finite difference operator
  \begin{equation*}
    B = \begin{bmatrix}
      1 & -1 \\ & \ddots & \ddots \\ & & 1 & -1
    \end{bmatrix}.
  \end{equation*}
  The entries of the noise vectors $\vec e$ are independently drawn from
  the standard normal distribution, after which the vector $\vec e$ is
  scaled such that $\epsilon = \mathbb{E}[\|\vec e\|] = 0.01\|\vec b\|$.
  We select the parameters $\mu$ such that $\|A\vec{x_\mu} - \vec b\| =
  \eta\epsilon$, where $\eta = 1 + 3.090232 / \sqrt{2n}$ so that $\|\vec
  e\|\le\eta\epsilon$ with probability $0.999$.
\end{example}



\begin{table}[!htbp]
  \footnotesize
  \centerfloat
  \caption{Truncated GSVD tests where only the nullspace of $B$ is
    deflated, and the iterations are terminated when the relative
    residual for the second largest generalized singular pair after
    $(1,0)$ is sufficiently small. The columns \textbf{Rank} and
    \textbf{Eff. cond} contain the numerical rank and effective
    condition number of $A$; and $\bm{\sin(\vec x_2, \vec{\widetilde
    x}_2)}$ is a measure for the error in the approximation of the
    generalized singular vector corresponding to the second largest
    generalized singular pair.}
  \label{tcgsvd:tab:tgsvd}
  \begin{tabular}{@{}lrr|rrr|rrr@{}}
    \toprule
    \textbf{Alg}
    & &
    & \multicolumn{3}{c|}{\textbf{GDGSVD}}
    & \multicolumn{3}{c}{\textbf{MDGSVD}} \\
    \textbf{Ex} & \textbf{Rank} & \textbf{Eff. cond}
    & $\bm{\sin(\vec x_2, \vec{\widetilde x}_2)}$ & \textbf{Rel. Err.} & \textbf{\#MV}
    & $\bm{\sin(\vec x_2, \vec{\widetilde x}_2)}$ & \textbf{Rel. Err.} & \textbf{\#MV} \\
    \midrule
    \textsf{Baart}     &   13 & $5.30\txte+12$ & $1.25\txte-5$ & $2.25\txte-5$ & 1632 & $3.40\txte-6$ & $2.08\txte-5$ &   72 \\
    \textsf{Deriv2-1}  & 1024 & $1.27\txte+06$ & $1.51\txte-5$ & $9.50\txte-5$ & 4228 & $7.92\txte-6$ & $2.75\txte-5$ & 2074 \\
    \textsf{Deriv2-2}  & 1024 & $1.27\txte+06$ & $1.51\txte-5$ & $8.61\txte-5$ & 4228 & $7.92\txte-6$ & $9.90\txte-6$ & 2074 \\
    \textsf{Deriv2-3}  & 1024 & $1.27\txte+06$ & $1.51\txte-5$ & $1.91\txte-3$ & 4228 & $7.92\txte-6$ & $1.22\txte-4$ & 2074 \\
    \textsf{Foxgood}   &   30 & $3.88\txte+12$ & $2.90\txte-5$ & $1.15\txte-5$ & 4148 & $4.59\txte-5$ & $2.31\txte-5$ & 2830 \\
    \textsf{Gravity-1} &   45 & $5.80\txte+12$ & $1.56\txte-5$ & $2.52\txte-4$ & 3764 & $1.06\txte-5$ & $9.97\txte-4$ & 1750 \\
    \textsf{Gravity-2} &   45 & $5.80\txte+12$ & $1.56\txte-5$ & $5.88\txte-4$ & 3764 & $1.06\txte-5$ & $9.85\txte-4$ & 1750 \\
    \textsf{Gravity-3} &   45 & $5.80\txte+12$ & $1.56\txte-5$ & $3.04\txte-4$ & 3764 & $1.06\txte-5$ & $3.77\txte-4$ & 1750 \\
    \textsf{Heat-1}    &  587 & $6.18\txte+12$ & $3.52\txte-5$ & $4.20\txte-2$ & 4976 & $9.48\txte-6$ & $7.16\txte-2$ &  802 \\
    \textsf{Heat-5}    & 1022 & $1.27\txte+03$ & $1.18\txte-5$ & $5.73\txte-2$ & 5036 & $7.44\txte-6$ & $1.28\txte-1$ &  616 \\
    \textsf{Phillips}  & 1024 & $2.90\txte+10$ & $1.25\txte-5$ & $5.82\txte-3$ & 4188 & $7.87\txte-6$ & $1.58\txte-3$ & 1762 \\
    \textsf{Shaw}      &   20 & $4.32\txte+12$ & $2.98\txte-5$ & $2.24\txte-2$ & 3644 & $1.32\txte-5$ & $7.75\txte-3$ & 2308 \\
    \textsf{Wing}      &    8 & $1.01\txte+12$ & $1.61\txte-5$ & $1.39\txte-5$ & 4492 & $4.55\txte-5$ & $1.44\txte-4$ & 3064 \\
    \bottomrule
  \end{tabular}
\end{table}



\begin{table}[!htbp]
  \footnotesize
  \centerfloat
  \caption{Truncated GSVD tests and results similar to
    Table~\ref{tcgsvd:tab:tgsvd}, but in this case with the
    approximation of the five largest generalized singular pairs after
    the pair $(1,0)$ corresponding to the nullspace of the
    regularization operator.}
  \label{tcgsvd:tab:defl}
  \begin{tabular}{@{}l|rrr|rrr@{}}
    \toprule
    \textbf{Alg}
    & \multicolumn{3}{c|}{\textbf{GDGSVD}}
    & \multicolumn{3}{c}{\textbf{MDGSVD}} \\
    \textbf{Ex}
    & $\bm{\sin(\vec x_2, \vec{\widetilde x}_2)}$ & \textbf{Rel. Err.} & \textbf{\#MV}
    & $\bm{\sin(\vec x_2, \vec{\widetilde x}_2)}$ & \textbf{Rel. Err.} & \textbf{\#MV} \\
    \midrule
    \textsf{Baart}     & $1.82\txte-6$ & $3.19\txte-6$ & 1996 & $2.61\txte-8$ & $1.51\txte-7$ &   74 \\
    \textsf{Deriv2-1}  & $8.03\txte-6$ & $8.99\txte-6$ & 6088 & $6.54\txte-6$ & $1.08\txte-5$ & 3604 \\
    \textsf{Deriv2-2}  & $8.03\txte-6$ & $3.52\txte-6$ & 6088 & $6.54\txte-6$ & $4.03\txte-6$ & 3604 \\
    \textsf{Deriv2-3}  & $8.03\txte-6$ & $6.25\txte-5$ & 6088 & $6.54\txte-6$ & $4.25\txte-5$ & 3604 \\
    \textsf{Foxgood}   & $6.91\txte-6$ & $3.36\txte-6$ & 6808 & $1.07\txte-5$ & $5.20\txte-6$ & 5485 \\
    \textsf{Gravity-1} & $1.93\txte-6$ & $1.14\txte-5$ & 5600 & $4.85\txte-6$ & $4.10\txte-5$ & 4012 \\
    \textsf{Gravity-2} & $1.93\txte-6$ & $3.11\txte-5$ & 5600 & $4.85\txte-6$ & $3.50\txte-5$ & 4012 \\
    \textsf{Gravity-3} & $1.93\txte-6$ & $8.39\txte-6$ & 5600 & $4.85\txte-6$ & $1.86\txte-5$ & 4012 \\
    \textsf{Heat-1}    & $2.70\txte-6$ & $2.82\txte-2$ & 7520 & $5.14\txte-6$ & $4.74\txte-2$ & 1948 \\
    \textsf{Heat-5}    & $7.92\txte-6$ & $4.63\txte-2$ & 6676 & $2.92\txte-6$ & $2.48\txte-2$ & 1804 \\
    \textsf{Phillips}  & $4.74\txte-6$ & $3.49\txte-4$ & 5912 & $2.30\txte-6$ & $1.63\txte-4$ & 3574 \\
    \textsf{Shaw}      & $1.91\txte-6$ & $6.51\txte-5$ & 5772 & $2.67\txte-6$ & $1.80\txte-4$ & 5620 \\
    \textsf{Wing}      & $8.33\txte-6$ & $4.16\txte-6$ & 5292 & $1.44\txte-5$ & $7.26\txte-6$ & 4618 \\
    \bottomrule
  \end{tabular}
\end{table}


Consider Example~\ref{tcgsvd:ex:tgsvd}, where we can write $\vec{x_\mu}$
as
\begin{equation*}
  \vec{x_\mu}
  = X (\Sigma_A^*\Sigma_A + \mu \Sigma_B^*\Sigma_B)^{-1} \Sigma_A^* U^* \vec b
  = \sum_{i=1}^n \frac{c_i}{c_i^2 + \mu s_i^2} \vec x_i \vec u_i^* \vec b.
\end{equation*}
For large-scale problems with rapidly decaying $c_i$ and multiple
right-hand sides $\vec b$, it may attractive to approximate the
truncated GSVD and compute the above summation only for a few of the
largest generalized singular pairs and their corresponding generalized
singular vectors.
In particular, we use our GDGSVD and MDGSVD methods to approximate the
truncated GSVD consisting of the 15 largest generalized singular pairs
and vectors.  We use minimum and maximum dimensions 15 and 45,
respectively, and a maximum of 100 restarts. We deflate or terminate
when the right-hand side of \eqref{tcgsvd:eq:cnvtol} is less than
$10^{-6}$, and seed the search space with the nullspace of $B$ spanned
by the vector $(1,\dots,1)^T$.  We consider two different cases. In the
first case, we deflate the seeded vector and terminate as soon as the
relative residual for the second largest generalized singular pair is
sufficiently small.  In the second case we deflate the seeded vector
plus four additional vectors, and terminate when the relative residual
corresponding to the sixth largest generalized singular pair is less
than $10^{-6}$.  We use the approximated truncated GSVDs to compute
$\vec{x_\mu}$, and compare it with the solution obtained with the exact
truncated GSVD.

The experiments are repeated with 1000 different initial vectors and
noise vectors, and the median results are reported in
Table~\ref{tcgsvd:tab:tgsvd} and Table~\ref{tcgsvd:tab:defl}.  Test
problems \textsf{Deriv2-\{1,2,3\}} all use the same matrix $A$, but have
different right-hand sides and solutions; the same is true for
\textsf{Gravity-\{1,2,3\}}. Test problems \textsf{Heat-\{1,5\}} have the
same solutions, but different $A$ and $\vec b$. The tables show a
reduction in the required number of matrix-vector products for
multidirectional subspace expansion, with reduction factors
approximately between 1.25 to 2.15 or better in the majority of cases.
However, the reduced number of matrix-vector products may come at the
cost of an increased relative error in the reconstructed solution and an
increased angle between the exact and approximated generalized singular
vector $\vec x_2$, although not consistently.



\section{Conclusion}
\label{tcgsvd:sec:con}

We have discussed two iterative methods for the computation of a few
extremal generalized singular values and vectors. The first method can
be seen as a generalized Davidson-type method, and the second as a
further generalization. Specifically, the second method uses
multidirectional subspace expansion combined with a truncation phase to
find improved search directions, while ensuring moderate subspace
growth. Both methods allow for a natural and straightforward thick
restart. We have also derived two different methods for the deflation of
generalized singular values and vectors.

We have characterized the locally optimal search directions and
expansion vectors in both the generalized Davidson method and the
multidirectional method. Note that these search directions generally
cannot be computed during the iterations. The inability to compute these
optimal search directions motivates multidirectional subspace expansion
and its reliance on the extraction process, as well as the removal of
low-quality search directions.  We have argued that our methods can
still achieve (asymptotic) linear convergence and have provided
asymptotic bounds on the rate of convergence. Additionally, we have
shown that the convergence of both methods is monotonic, and have
concluded the theoretical analysis by developing Rayleigh--Ritz theory
and generalizing known results for the Hermitian eigenvalue problem to
the Hermitian positive definite generalized eigenvalue problem that
corresponds to the GSVD.

The theoretical convergence behavior is supported by our numerical
experiments. Moreover, the numerical experiments demonstrate that our
generalized Davidson-type method is competitive with existing methods,
and suitable for approximating the truncated GSVD of matrix pairs with
rapidly decaying generalized singular values. Significant additional
performance improvements may be obtained by our new multidirectional
method.



%


\printbibliography

\end{document}